\documentclass[a4paper,leqno]{article}
\usepackage[a4paper,centering,vscale=0.77,hscale=0.77]{geometry}
\usepackage{amsmath,amsthm,amssymb,bbm,bm}
\usepackage{xcolor}
\usepackage{verbatim}
\usepackage{todonotes}
\usepackage{mathrsfs}

\usepackage{thmtools}
\usepackage[pdfdisplaydoctitle,colorlinks,breaklinks,urlcolor=blue,linkcolor=blue,citecolor=blue]{hyperref}
\usepackage{enumitem}

\newtheorem{theorem}{Theorem}[section]
\newtheorem{lemma}[theorem]{Lemma}
\newtheorem{proposition}[theorem]{Proposition}

\newtheorem{definition}[theorem]{Definition}
\newtheorem{remark}[theorem]{Remark}

\def\enne{\mathbb{N}}
\def\zeta{\mathbb{Z}}

\def\erre{\mathbb{R}}

\def\P{\mathbb{P}}

\def\E{\mathop{{}\mathbb{E}}}

\def\LL{\mathcal{L}}
\def\cF{\mathscr{F}}
\def\cB{\mathscr{B}}

\def\A{\mathcal{A}}
\def\dd{\mathsf{d}}

\def\cP{\mathscr{P}}

\renewcommand{\d}{{\mathrm d}}

\numberwithin{equation}{section}

\def\beq{\begin{equation}}
\def\eeq{\end{equation}}

\def\to{\rightarrow}

\def\embed{\hookrightarrow}

\def\norm #1{\left\|#1\right\|}

\def\sp #1#2{\left<#1,#2\right>}
\newcommand\ip\sp

\begin{document}

\title{Degenerate Kolmogorov equations and ergodicity \\
for the stochastic Allen-Cahn equation\\ with logarithmic potential}
\author{Luca Scarpa\thanks{Department of Mathematics, Politecnico di Milano,
Via E.~Bonardi 9, 20133 Milano, Italy.
E-mail: \texttt{luca.scarpa@polimi.it}
URL: \texttt{https://sites.google.com/view/lucascarpa}}
\and
Margherita Zanella\thanks{Department of Mathematics, Politecnico di Milano,
Via E.~Bonardi 9, 20133 Milano, Italy.
E-mail: \texttt{margherita.zanella@polimi.it}}}

\maketitle

\begin{abstract}
  Well-posedness \`a la Friedrichs is proved for a class of degenerate 
  Kolmogorov equations associated to stochastic Allen-Cahn equations
  with logarithmic potential.
  The thermodynamical consistency of the model requires the 
  potential to be singular and the multiplicative noise 
  coefficient to vanish at the respective potential barriers, 
  making thus the corresponding Kolmogorov equation 
  not uniformly elliptic in space.
  First, existence and uniqueness of
  invariant measures and ergodicity are discussed. Then,
  classical solutions to some regularised Kolmogorov equations
  are explicitly constructed. Eventually, 
  a sharp analysis of the blow-up rates of the regularised solutions and 
  a passage to the limit with a specific scaling yield
  existence \`a la Friedrichs for the original Kolmogorov equation.
  
  \smallskip
  
  
  \smallskip
  
\end{abstract}

\tableofcontents

\section{Introduction}
Modelling the evolution of multiphase materials 
- e.g.~binary fluid mixtures, metallic alloys,
heterogenous human tissues - has become fundamental in the last decades in 
numerous fields such as Material Science, Biology, and Engineering.
One of the well-established mathematical ways of describing phase-separation 
is the so-called {\em diffuse interface}, or {\em phase-field}, approach. This consists in introducing 
a phase-variable $u$, or order parameter, with values in $[-1,1]$: the regions 
$\{u=1\}$ and $\{u=-1\}$ represent the pure phases, and it is assumed that there is 
a narrow blurred interfacial layer in between, where $u$ can take also the intermediate values $(-1,1)$. 
Such description has been firstly proposed by Cahn and Hilliard \cite{cahn-hill}
to model conserved dynamics of spinodal decomposition in metallic alloys, 
and since then has been extensively employed in several contexts.

One of the classical phase-field models for non-conserved
phase-separation is the Allen-Cahn equation:
this has been originally introduced in the context of Van oder Waals theory of phase transition
and has then been employed by Allen and Cahn in \cite{AC} 
for describing growth of grains in crystalline materials close to their melting points.
In its classical form, the deterministic Allen-Cahn equation reads
\beq\label{det_AC}
  \partial_t u - \nu\Delta u + F'(u) = f \qquad\text{in } (0,T)\times D,
\eeq
where $D$ is a smooth bounded domain in $\erre^d$ ($d=2,3$), $T>0$ is a given reference time, 
$f$ is a suitable forcing term, and $\nu>0$ is a given constant depending on the structural data
such as the thickness of the separation layer.
The equation is usually complemented with a given initial datum, and 
homogeneous boundary conditions of Neumann or Dirichlet type.
The nonlinearity $F'$ represents the derivative of a double-well potential $F$,
which is required to be singular at $\pm1$ by the thermodynamical consistency of the model:
the relevant choice for $F$ is indeed the so-called Flory-Huggins logarithmic potential \cite{Flory} given by
\beq
  \label{F_log}
  F_{log}(r):=\frac\theta2\left[(1+r)\ln(1+r) + (1-r)\ln(1-r)\right] - \frac{\theta_0}2r^2, 
  \quad r\in(-1,1),
\eeq
where $0<\theta<\theta_0$ are fixed constant related to the critical temperature of the material in consideration.
Note that $F_{\log}$ is continuous on $[-1,1]$, with two global minima in $(-1,1)$, while $F_{log}'$ blows up 
at the potential barriers $\pm1$. This is coherent with the physical interpretation of 
diffuse-interface modelling in which only the values of the variable $u \in [-1,1]$ are meaningful.
The Allen-Cahn equation can also be seen as the gradient flow with respect to the $L^2(D)$-metric of the 
associated free-energy functional 
\begin{equation}
\label{energy}
\mathcal{E}(u):= \int_D\left(\frac{\nu}{2}|\nabla u|^2+ F(u)\right),
\end{equation}
where the former energy contribution penalises for high oscillations of $u$ while the 
latter takes into account the typical mixing/demixing effects.

Due to the singularity of the derivative $F'$, 
for mathematical simplicity the double-well potential $F$ is often approximated by a smooth
one in polynomial form. Let us stress that although this may be useful in the mathematical
treatment of the equation, it is a severe drawback on the modelling side:
for example, such choice does not even ensure the preservation of the physically relevant bound $u\in[-1,1]$
in general. For this reason, throughout the paper we deal only with
thermodynamically relevant potentials such as the logarithmic one \eqref{F_log},
as required by the model.

The deterministic Allen-Cahn equation provides a good description of the evolution of the phase separation.
Nonetheless, it presents some disadvantages. Indeed, it is not general enough to capture 
possible unpredictable effects which may affect phase-separation, such as 
thermal fluctuations, magnetic disturbances, or microscopic configurational phenomena.
These can be taken into account by adding a Wiener noise in the equation,
as suggested originally in the well celebrated stochastic Cook model for phase-separation \cite{cook}
and then confirmed in several contributions (see e.g.~\cite{BGW10, BMW08}).
By allowing for a stochastic Wiener-type forcing in \eqref{det_AC}, we deal with 
the stochastic Allen-Cahn equation in the general form
\begin{equation}
\label{eq_AC}
\begin{cases}
{\rm d}u-\nu\Delta u\,{\rm d}t +F'(u)\,{\rm d}t=
B(u)\,{\rm d}W & \text{in } (0,T) \times D,
\\
\alpha_du + \alpha_n\partial_{\bf n}u=0 & \text{in } (0,T) \times \partial D,
\\
u(0)=u_0 & \text{in } D,
\end{cases}
\end{equation}
where $W$ is a cylindrical Wiener process defined on a certain separable Hilbert space and $B$ is a suitable stochastically integrable operator with respect to $W$. The parameters $\alpha_d,\alpha_n\in\{0,1\}$
are such that $\alpha_d+\alpha_n=1$ and are thus responsible for the choice of 
Dirichlet or Neumann boundary conditions.

In the case of a logarithmic relevant potential \eqref{F_log},
well-poseness for the stochastic Allen-Cahn equation with Neumann boundary conditions
has been addressed for the first time in the very recent contribution \cite{Ber}.
Qualitative studies on the associated random separation principle have then been
analysed in \cite{BOS}. Roughly speaking, the novel idea to overcome the singularity 
of $F'$ was to employ a degenerate noise coefficient $B$ that vanishes at the potential barriers $\pm1$
in such a way to compensate the blow up of $F''$: existence of analytically strong solutions
(see Definition~\ref{strong_sol_def} below) is obtained for initial data satisfying 
\[
  u_0 \in V\cap\A, \quad\A:=\left\{v\in L^2(D): |v({\bf x})|\leq1 \quad\text{for a.e.~}{\bf x}\in D\right\},
\]
where $V$ is either $H^1(D)$ or $H^1_0(D)$, depending on the boundary condition.
The method is quite robust, in the sense that 
it has been applied also to different singular phase-filed type equations: let us mention,
above all, the contributions \cite{DGG, FG} on the stochastic thin-film equation,
\cite{S-SCH} on the stochastic Cahn-Hilliard equation with degenerate mobility,
and \cite{bauz-bon-leb} on the stochastic Allen-Cahn equation with single obstacle potential.

In general, the mathematical literature on stochastic phase-filed models
is becoming increasingly popular, both in the analytical and probabilistic communities.
We refer, for example, to the works \cite{HRW, orr-scar} on the stochastic Allen-Cahn equation,
and to \cite{daprato-deb, S-SCH2, scar-OVCSCH} on the stochastic Cahn-Hilliard equation,
as well as to the references therein.

The aim and novelty of the present paper is to investigate the elliptic 
Kolmogorov equations associated to the stochastic dynamics given by \eqref{eq_AC}
on the Hilbert space $H:=L^2(D)$.
The motivations are numerous. In particular, the Kolmogorov equation is intrinsically connected
with the long-time behaviour of solutions and ergodicity of the stochastic system \eqref{eq_AC}.
Indeed, provided to prove existence of invariant measures for the associated transition semigroup, 
the Kolmogorov operator is the natural candidate to be its respective infinitesimal generator.

For the stochastic Allen-Cahn equation \eqref{eq_AC}, 
setting $Q:=BB^*$ the Kolmogorov equation reads
\beq
  \label{eq_K}
  \alpha\varphi(x) - \frac12\operatorname{Tr}[Q(x)D^2\varphi(x)]
  +\left(-\Delta x + F'(x), D\varphi(x)\right)_H = g(x), \qquad
  x\in\A_{str},
\eeq
where $\alpha$ is a fixed positive constant, $g\in C^0_b(H)$ is a given forcing, and
\[
  \A_{str}:=\left\{v\in V\cap\A: -\Delta v + F'(v) \in H\right\}.
\]
Note that the nonlinear condition on $x\in\A_{str}$ is necessary.
Indeed, the singularity of the derivative $F'$ in \eqref{F_log}
forces the solution $u$ to take values in $(-1,1)$: consequently, the respective Kolmogorov 
equation \eqref{eq_K} only makes sense on the bounded subset $\A_{str}$ of $H$.

The main severely pathological behaviour of equation \eqref{eq_K} is that 
the second-order diffusion operator is {\em not} uniformly elliptic in space:
this is due to the degeneracy of $B$ at the boundary $\partial\A$, which is needed
in order to solve the SPDE \eqref{eq_AC}, as pointed out above. Of course, 
such degeneracy has important consequences on the mathematical analysis of 
\eqref{eq_K}, as in general one cannot expect to obtain solutions with some reasonable 
space-regularity. This inevitably calls for the introduction of weaker notions of solutions 
which are better suited to incorporate such lack of control: the idea is to employ 
so-called solutions {\`a la Friedrichs} (see Proposition~\ref{prop:class} below), 
which are defined, {\em roughly speaking}, as limits of classical solutions in suitable topologies.
Clearly, one needs to properly identify which is the natural functional setting in order to 
pass to the limit. In this direction, a preliminary study on long-time behaviour and ergodicity 
for the transition semigroup associated to \eqref{eq_AC} reveals that
every invariant measure is concentrated on the bounded subset $\A_{str}$.
This suggests that the natural functional setting that allows to pass to the limit 
in the sense of Friedrichs is the one of Lebesgue spaces 
associated to some invariant measure for the SPDE \eqref{eq_AC},
since invariant measures for \eqref{eq_AC} basically ``ignore'' the behaviour 
of $\varphi$ outside $\A_{str}$.

The second difficulty that comes in play concerns the multiplicative nature of the covariance operator $Q$.
Indeed, in order to pass to the limit in  the sense of Friedrichs, one has to 
sharply balance the convergence of some regularised operators $Q_{\lambda,n}$ to $Q$
with the explosion of the second derivatives of the respective classical solutions $\varphi_{\lambda,n}$,
as the regularisation parameters $\lambda$ and $n$ vanish.
Intuitively speaking, if one is able to show that the convergence rate of 
$Q_{\lambda,n}-Q$ dominates the explosion rate of $D^2\varphi_{\lambda,n}$,
a passage to the limit yields existence of solutions for the limit Kolmogorov equation \eqref{eq_K}
\`a la Friedrichs. Of course, this calls for a sharp analysis on the explosion and convergence 
rates of the approximating classical solutions with respect to their respective regularising parameters.

The literature on long-time behaviour and ergodicity for stochastic systems is extremely developed.
A very general study on ergodicity and Kolmogorov equations for 
stochastic evolution equations in variational form with additive noise was
carried out by Barbu and Da Prato \cite{BDP} in a very general setting.
Still in the framework of variational approach to ergodicity of SDPEs, we can 
mention the contributions \cite{MZ} on Poisson-type noise and \cite{MS-erg}
in the case of semilinear equations with singular drift.
An extensive literature on ergodicity and Kolmogorov equations 
in the mild setting has been growing in the last decades, for which we refer to 
the works \cite{cerrai, Dap, DapZab3, stann}.
In particular, in the context of semilinear reaction-diffusion equations 
existence and uniqueness of invariant measures, as well as moment estimates, 
are obtained in \cite{stann2, stann3}. For stochastic porous media equations
we refer to the recent contribution \cite{DGT}. Ergodicity for stochastic 
damped Schr\"odinger equation has been studied in \cite{DO, BFZ, BFZ2}, while 
long-time behaviour for Euler- and Navier-Stokes-type equations has been 
addressed, among many others, in \cite{BF, DD, GMR, HM, RZ}.

Concerning Kolmogorov equations with degenerate covariance operator $Q$, 
well-posedness results are significantly less developed. 
To the best of our knowledge, the main available contributions so far concern the 
parabolic Kolmogorov equations associated to semilinear stochastic equations:
through the notion of generalised solutions 
existence is obtained ``by hand'' via regular dependence of the SPDE on the initial datum,
by exploiting some suitable smoothness assumptions on the nonlinearities in play.
For further detail we refer the reader to \cite[Sec.~7.5]{DapZab2}.
In the same spirit, parabolic Kolmogorov equations associated to stochastic PDEs with 
multiplicative noise are dealt with in \cite{dap_mult}, still under appropriate
smoothness requirements on the coefficients or nondegeneracy conditions on the covariance.
More generally, the study of Kolmogorov equations associated to 
stochastic PDEs has become crucial in the last years in  the direction of
uniqueness and  regularisation by noise. 
Let us point out, above all, the recent contributions \cite{mau} on non-explosion 
for SDEs via Stratonovich noise and \cite{AMP} on a BSDE approach to uniqueness by noise.

Let us conclude by briefly summarise the content of the paper.
In Section~\ref{sec:frame} we introduce the mathematical setting, state the main assumptions, and
recall the available well-posedness results.
Section~\ref{sec:inv} is devoted then to the study 
of invariant measures and ergodicity for equation \eqref{eq_AC}: in particular, 
we show existence of (possibly ergodic and strongly mixing) invariant measures,
we provide sufficient conditions for uniqueness, and we characterise their support.
In Section~\ref{sec:kolm} we focus  on the Kolmogorov equation associated to \eqref{eq_AC}.
In particular, we first introduce the Kolmogorov operator, 
as well as some suitable regularised Kolmogorov equations,
depending on two approximating parameters. Secondly, we construct classical solutions
to such regularised equations ``by hand'', by exploiting appropriate 
regular dependence on the initial data for the corresponding regularised SPDEs. Eventually, 
we obtain uniform estimates on the approximated solutions and 
sharp blow-up rates on their derivatives, allowing us to prove existence of solution
for the original equation \eqref{eq_K} through a passage to the limit 
on a specific scaling of the parameters. This shows well-posedness \`a la Friedrichs 
for the Kolmogorov equation, and characterises the Kolmogorov operator as the 
infinitesimal generator of the transition semigroup in some Lebesgue space
associated to some invariant measure.
Eventually, Appendices ~\ref{app1}--\ref{app2} contain useful estimates on the 
stochastic Allen-Cahn equation \eqref{eq_AC} and a density result used in the proofs, respectively.


\section{Mathematical framework}
\label{sec:frame}
\subsection{Notation and setting}
For any real Banach space $E$, we denote its dual by $E^*$. 
The duality pairing between $E$ and $E^*$ will be indicated by $\langle \cdot, \cdot\rangle_E$. 
For any real Hilbert space $H$ we denote by $\|\cdot\|_H$ and $(\cdot, \cdot)_H$ 
the norm and the scalar product respectively. Given any two Banach spaces 
$E$ and $F$, we use the symbol $\mathcal{L}(E,F)$ for the space of all linear bounded operators form $E$ to $F$. Furthermore, we write $E \hookrightarrow F$, if $E$ is continuously embedded in $F$.
If $H$ and $K$ are separable Hilbert spaces, we employ the symbol
$\LL_{HS}(H, K)$ for the space of Hilbert-Schmidt operators from $H$ to $K$. 
For any topological space $E$, the Borel $\sigma$-algebra on $E$ 
is denoted by $\mathcal{B}(E)$. All measures on $E$ are intended to be defined on its Borel $\sigma$-algebra. The spaces of bounded Borel-measurable and 
bounded continuous functions on $E$ will be denoted by $B_b(E)$ and $C_b^0(E)$ respectively.

If $(A, \mathcal{A}, \mu)$ is a finite measure space, we denote by $L^p(A;E)$ 
the space of $p$-Bochner integrable functions, for any $p \in [1, \infty)$. 
For a fixed $T>0$, we denote by $C^0([0,T];E)$ the space of strongly continuous functions from $[0,T]$ to $E$. 

If quantities  $a,b\ge 0$ satisfy the inequality $a \le C(A) b$ with a 
constant $C(A)>0$ depending on the expression $A$, we write $a \lesssim_A b$; 
for a generic constant we put no subscript.
If we have $a \lesssim_A b$ and $b \lesssim_A a$, we write $a \simeq_A  b$.

Throughout the paper, $D \subset \mathbb{R}^d$, $d=2,3$, is a bounded domain with Lipschitz boundary $\Gamma$
and Lebesgue measure denoted by $|D|$. The coefficients 
$\alpha_d,\alpha_n\in\{0,1\}$ are such that $\alpha_d+\alpha_n=1$:
the case $(\alpha_d,\alpha_n)=(1,0)$ corresponds to Dirichlet boundary conditions, 
while $(\alpha_d,\alpha_n)=(0,1)$ yields Neumann boundary conditions.
We introduce the functional spaces
\[
  H:=L^2(D)
\]
and
\begin{align*}
  V&:=
  \begin{cases}
  H^1_0(D) \quad&\text{if } (\alpha_d,\alpha_n)=(1,0),\\
  H^1(D) \quad&\text{if } (\alpha_d,\alpha_n)=(0,1),
  \end{cases}
  \\
  Z&:=
  \begin{cases}
  H^2(D)\cap H^1_0(D) \quad&\text{if } (\alpha_d,\alpha_n)=(1,0),\\
  \{v\in H^2(D):\partial_{\bf n}v=0 \text{ a.e.~on } \Gamma\} \quad&\text{if } (\alpha_d,\alpha_n)=(0,1).
  \end{cases}
\end{align*}
all endowed with their natural respective norms $\|\cdot\|_H$, $\|\cdot\|_V$, and $\|\cdot\|_{Z}$.
Identifying the Hilbert space $H$ with its dual through the Riesz isomorphism, 
we have the following continuous, dense and compact inclusions
\begin{equation*}
Z \hookrightarrow V \hookrightarrow H \simeq H^* \hookrightarrow V^* \hookrightarrow Z^*.
\end{equation*} 
In particular, $(V,H,V^*)$ constitutes a Gelfand triple. 
The norm of the continuous inclusion $V\embed H$
will be denoted by $K_0$: note that $K_0$ can be estimated by means 
of Poincar\'e-type inequalities in terms of the first positive eigenvalue of the Laplacian.

We recall that the Laplace operator with homogeneous (Dirichlet or Neumann) conditions can be seen either as a variational operator
\begin{equation*}
-\Delta\in\LL(V,V^*), \qquad \langle -\Delta u, v\rangle_V:= \int_D \nabla u \cdot\nabla v, \quad  u,v \in V,
\end{equation*}
or as an unbounded linear operator on $H$ with effective domain $Z$.
In the sequel we will use the same symbol $-\Delta$ to denote the Laplace operator 
intended both as a variational operator and as an operator defined from $Z$ with values in $H$.

Let $(\Omega,\mathcal{F},\mathbb{P})$ be a probability space, 
$U$ a separable real Hilbert space, with a given orthonormal basis $(e_k)_{k\in\mathbb{N}}$, 
and $W$ a canonical cylindrical Wiener processes taking values in $U$
and adapted to a filtration $\mathbb{F}$ satisfying the usual conditions.
Given $p,q\in[1,+\infty)$, $T>0$, and a Banach space $E$,
we denote by the symbol  $L^p(\Omega; L^q(0,T; E))$ the space of $E$-valued progressively 
measurable processes $X:\Omega\times(0,T)\to E$ such that 
$\E(\int_0^T\norm{X(s)}_E^q\,\d s)^{p/q}<+\infty$. When 
$E$ is a separable Hilbert space, $p\in(1,+\infty)$, and $q=+\infty$, 
the symbol $L^p(\Omega; L^\infty(0,T; E^*))$
denotes the space of weak star measurable random variables 
$X:\Omega\to L^\infty(0,T; E^*)$ such that $\E\norm{X}_{L^\infty(0,T; E^*)}^p<+\infty$,
which by \cite[Thm.~8.20.3]{edwards} is isomorphic to the dual of 
$L^\frac{p}{p-1}(\Omega; L^1(0,T; E))$.

\subsection{Assumptions}
Let us state the set of Assumptions that will be used throughout the paper. 
We work in an analogous framework as the one of \cite{Ber}.

\begin{description} 
\item[H1] The potential $F:[-1,1]\to[0,\infty)$ satisfies the following conditions:
  \begin{description}
  \item[(i)] $F\in C^0([-1,1])\cap C^3(-1,1)$ and $F'(0)=0$,
  \item[(ii)] there exists $K>0$ such that $F''(r)\geq -K$ for all $r\in(-1,1)$,
  \item[(iii)] it holds that 
  \[
  \lim_{r\to(\pm1)^{\mp}}F'(r)=\pm\infty.
  \]
  \end{description}
In this setting, note that conditions {\bf(i)--(iii)} ensure the existence of
constants $C_0,C_1>0$ such that 
\begin{equation}
\label{F'_prop}
F'(r)r\ge C_0r^2-C_1.
\end{equation}
It is straightforward to see that the logarithmic potential \eqref{F_log}
(up to some additive constant) satisfies conditions {\bf (i)--(iii)}.

\item[H2] Let $\{h_k\}_{k \in \mathbb{N}} \subset C^1([-1,1])$ satisfy 
for every $k \in \mathbb{N}$ that $h_k(\pm 1)=0$ and 
\begin{equation}
\label{C_B}
C_B := \sum_{k \in \mathbb{N}}
\left(\|h_k\|^2_{C^1([-1,1])} +
\norm{h_k^2F''}_{L^\infty(-1,1)} \right) < \infty.		
\end{equation}
Setting $\A:= \{v \in H:|v({\bf x})| \le 1 \text{ for a.e.~}{\bf x}\in D\}$, 
condition \eqref{C_B} implies that
the operator 
\begin{equation}
\label{B}
B:\A \rightarrow \mathcal{L}_{HS}(U,H), \qquad
B(x)e_k:=h_k(x), \quad x \in \A, \quad k \in \mathbb{N},
\end{equation}
is well-defined and Lipschitz-continuous. Indeed, this amounts 
to saying that 
\[
B(x)e:= \sum_{k \in \mathbb{N}}(e,e_k)_Uh_k(x), \quad x \in \A, \ e \in U,
\]
and \eqref{C_B} yields by a direct computation (see e.g.~\cite[Sec.~2]{Ber}) that 
\begin{align}
  \label{HS_norm}
  \|B(x)\|_{\mathcal{L}_{HS}(U,H)}^2 &\le C_B|D| \qquad\forall\, x \in \A,\\
  \|B(x)-B(y)\|_{\mathcal{L}_{HS}(U,H)}^2 &\le C_B|D|\|x-y\|^2_H \qquad\forall\, x, y \in \A.
\end{align}
\end{description}

\subsection{Well posedness results}

The existence and uniqueness of solutions to problem \eqref{eq_AC} is proved in \cite{Ber}
in the case of Neumann boundary conditions and exclusively for 
relevant case of logarithmic potential \eqref{F_log}.
One can easily check that the same results hold true in the case of Dirichlet boundary conditions
and under the more general assumption {\bf H1} for $F$, by using \eqref{C_B} (see e.g.~\cite{S-SCH} for details).
We recall here the main well posedness results.

\begin{definition}
\label{var_sol_def}
Let 
\begin{equation}
  \label{u0_var}
  u_0 \in L^2(\Omega,\cF_0; H), \qquad \P\{u_0\in \A\}=1.
\end{equation}
A variational solution to problem \eqref{eq_AC} is a process $u$ such that, for every $T>0$,
\begin{align}
\label{reg_u}
u &\in L^2(\Omega;C([0,T];H))\cap L^2(\Omega;L^2(0,T;V)),\\
\label{reg_F'}
F'(u) &\in L^2(\Omega;L^2(0,T;H)), 
\end{align}
and for all $\psi \in V$ it holds that, for every $t \ge0$, $\mathbb{P}$-a.s., 
\begin{align}
\label{var_sol}
\int_D u(t)\psi  &+ \nu\int_0^t \int_D \nabla u(s)\cdot\nabla \psi(s)\, {\rm d}s 
+ \int_0^t \int_D F'(u(s))\psi\, {\rm d}s
\notag\\
&= \int_D u_0\psi 
+ \int_D\left( \int_0^t B(u(s))\,{\rm d}W(s)\right) \psi.
\end{align}
\end{definition}

\begin{definition}
\label{strong_sol_def}
Let 
\begin{equation}
  \label{u0_strong}
  u_0 \in L^2(\Omega,\cF_0; V), \qquad \P\{u_0\in \A\}=1.
\end{equation}
An analitically strong solution to problem \eqref{eq_AC} is a process $u$ such that, for every $T>0$,
\begin{align}
\label{reg_u_str}
u &\in L^2(\Omega;C([0,T];H))\cap L^2(\Omega;L^{\infty}(0,T;V))\cap L^2(\Omega;L^2(0,T;Z)),\\
\label{reg_F'_str}
F'(u) &\in L^2(\Omega;L^2(0,T;H)), 
\end{align}
and it holds that, for every $t \ge0$, $\mathbb{P}$-a.s., 
\begin{align}
\label{strong_sol}
u(t)-\nu\int_0^t\Delta u(s)\, {\rm d}s + \int_0^t F'(u(s))\, {\rm d}s
&=u_0 +  \int_0^t B(u(s))\,{\rm d}W(s).
\end{align}
\end{definition}

The well-posedness result following from \cite[Thm.~2.1]{Ber} is the following.
\begin{theorem}
\label{ex_uniq_sol}
Assume {\bf H1--H2}.
For every $u_0$ satisfying \eqref{u0_var} 
there exists a unique variational solution to \eqref{eq_AC} in the sense of Definition~\ref{var_sol_def}. 
Furthermore, for every $T>0$ there exists a positive constant $C_T$ such that, 
for every initial data $u_0^1, u_0^2$ satisfying \eqref{u0_var} 
the respective variational solutions $u_1, u_2$ of \eqref{eq_AC} satisfy
\begin{equation}
\label{difference}
\|u_1-u_2\|_{L^2(\Omega;C([0,T],H)) \cap L^2(\Omega;L^2(0,T,V))} \le C_T
\|u_0^1-u_0^2\|_{L^2(\Omega;H)}.
\end{equation}
Moreover, for every $u_0$ satisfying \eqref{u0_strong}
there exists a unique analytically strong solution to \eqref{eq_AC} in the sense of Definition~\ref{strong_sol_def}.
\end{theorem}

\section{Invariant measures}
\label{sec:inv}
This section is devoted to the long-time analysis of 
the stochastic equation \eqref{eq_AC}, in terms of existence-uniqueness
of invariant measures and ergodicity.
Before moving on, we recall some general definitions that will be used in the sequel.

For every $x\in \A$, the unique variational solution 
to equation \eqref{eq_AC} as given in Theorem~\ref{ex_uniq_sol} will be denoted by $u^x$,
and for every $t\in[0,T]$ we set  $u(t;x):=u^x(t)$ for its value at time $t$.
Note that for every $t\in[0,T]$ $u(t;x):\Omega\to H$ is a random variable in $L^2(\cF_t; H)$.

The main issue in defining the concept of invariant measure in our framework 
is that equation \eqref{eq_AC} can be solved only if the initial datum 
satisfies a nonlinear type condition (see \eqref{u0_var}).
In this direction, it is useful to extend $F$ to $+\infty$ outside $[-1,1]$,
and obtain a proper convex lower semicontinuous function $F:\erre\to[0,+\infty]$.
With this notation, we have the characterisation (see again {\bf H2})
\begin{equation}
\label{mat_A}
\mathcal{A}=\{x \in H : \|x\|_{L^\infty(D)}\leq 1\} = \{x \in H : F(x)\in L^1(D)\}.
\end{equation}
We claim that $\mathcal A$ is a Borel subset of $H$. Indeed, one has that 
\[
  \mathcal A=\bigcup_{n\in\enne}\left\{x\in H: \ \int_DF(x) \leq n\right\},
\]
where the right-hand side is a countable union of closed sets in $H$
by lower semicontinuity of $F$, hence is a Borel subset of $H$.
The equality in \eqref{mat_A} follows from the fact that the domain of $F$ is exactly $[-1,1]$.

We consider on $\mathcal{A}$ the metric $\dd$ given by the restriction to $\mathcal{A}$ 
of the metric on $H$ induced by the $H$-norm. Being $\mathcal{A}$ a closed subspace 
of the complete separable metric space $(H, \|\cdot\|_H)$, $(\mathcal{A},\dd)$ is also complete and separable.

The space $(\mathcal{A}, \dd)$ is therefore a separable complete metric space. We denote by 
$\cB(\mathcal{A})$ the $\sigma$-algebra of all Borel subsets of $\mathcal{A}$ and by $\cP(\mathcal{A})$ 
the set of all probability measures on $(\mathcal{A}, \cB(\mathcal{A}))$. 
Also, the symbol $\cB_b(\A)$ denotes the space of Borel measurable bounded functions from $\A$ to $\erre$. 
If $A \in \cB(\mathcal{A})$, we denote by $A^C$ its complement.

With this notation and by virtue of Theorem~\ref{ex_uniq_sol}, we can introduce the family of operators
$P:=(P_t)_{t \ge 0}$ associated to equation \eqref{eq_AC} as 
\begin{equation}
\label{P_t}
(P_t\varphi)(x):= \mathbb{E}[ \varphi(u(t;x))], \quad x\in\A,\quad \varphi\in \cB_b(\mathcal{A}).
\end{equation}

\begin{remark}
Let us point out once more that, due to the nonlinear nature of the problem, 
the solution of equation \eqref{eq_AC} exists on $\mathcal{A}$, hence
the transition semigroup can only make sense as a family of operators acting on $\cB_b(\mathcal{A})$,
and not on $\cB_b(H)$ as in more classical cases.
\end{remark}

It is clear that $P_t\varphi$ is bounded 
for every $\varphi \in \cB_b(\mathcal{A})$. We know from \cite[Cor.~23]{On2005}
that the transition function is jointly measurable, that is for any $A\in\cB(\mathcal{A})$ 
the map $\mathcal{A} \times [0,\infty)\ni  (x,t)\mapsto \mathbb{P}\{u(t;x)\in A\}\in \mathbb{R}$ is measurable. So $P_t\varphi$ is also measurable for every $\varphi \in \mathcal{B}_b(\mathcal{A})$, 
hence $P_t$ maps $\cB_b(\mathcal{A})$ into itself for every $t \ge0$.
Furthermore, since the unique solution of \eqref{eq_AC} is an $H$-valued continuous process, 
then it is also a Markov process, see \cite[Theorem 27]{On2005}. 
Therefore we deduce that the family of operators $\{P_t\}_{t\ge0}$ is a Markov semigroup, 
namely $P_{t+s}=P_tP_s$ for any $s,t\ge0$.

We are ready to give the precise definition of invariant measure.
\begin{definition}
  An invariant measure for the transition semigroup $P$ 
  is a probability measure $\mu\in\cP(\A)$ such that 
  \[
  \int_\A \varphi(x)\,\mu(\d x) = \int_\A P_t\varphi(x)\,\mu(\d x) \quad\forall\,t\geq0,\quad\forall\,\varphi\in C_b(\A).
  \]
\end{definition}

\subsection{Existence of an invariant measure}
We focus here on showing that $P$ admits at least an invariant measure.
The main idea is to use an adaptation of the Krylov-Bogoliubov theorem
to the case of complete separable metric spaces, which we prove here for clarity.
The proof is an adaptation of the one 
in the more classical Hilbert space setting, which can be found in \cite[Thm.~11.7]{DapZab}.

\begin{theorem}[Krylov-Bogoliubov]
\label{KriBou}
Let $R:=\{R_t\}_{t\ge 0}$ be a time-homogeneous Markov semigroup 
on the complete separable metric space $(\mathcal{A}, \dd)$. Assume that
\\
i) the semigroup $\{R_t\}_{t\ge 0}$ is Feller in $\mathcal{A}$;
\\
ii) for some $x_0 \in \mathcal{A}$, the set $(\mu_t)_{t>0}\subset \cP(\mathcal{A})$ given by
\begin{equation}
\label{mu_t}
\mu_t (A):= \frac 1t \int_0^t (R_s\pmb{1}_{A})(x_0)\, {\rm d}s, \qquad A \in \cB(\mathcal{A}), \ t>0,
\end{equation}
is tight. Then there exists  at least one invariant measure for $R$.
\end{theorem}
\begin{proof}
By the Prokhorov Theorem (see e.g. \cite[Vol.~II, Thm~8.6.2]{Bog}) 
there exists a subsequence $\{t_n\}_n$ with $\lim_{n\to\infty}t_n = \infty$ 
and a probability measure $\mu \in \cP(\mathcal{A})$ such that
\begin{equation*}
\lim_{n \rightarrow \infty}\int_{\mathcal{A}}
 \varphi(x)\,\mu_{t_n}(\d x) = \int_{\mathcal{A}}\varphi(x)\, \mu(\d x) \qquad \forall\, \varphi \in C_b(\mathcal{A}).
\end{equation*}
By \eqref{mu_t} and the Fubini Theorem the above expression is equivalent to
\begin{equation}
\label{lim0}
\lim_{n \rightarrow \infty}\frac{1}{t_n}
\int_0^{t_n} (R_t\varphi)(x_0)\,{\rm d}t = 
\int_{\mathcal{A}}\varphi(x)\, \mu(\d x) \qquad \forall \, \varphi \in C_b(\mathcal{A}).
\end{equation}
Given $s\ge 0$ and $\psi\in C_b(\A)$, we have that $R_s\psi\in C_b(\A)$ by the Feller property.
Hence, we can choose $\varphi=R_s \psi$ in \eqref{lim0} and infer that
\begin{equation}
\label{lim}
\lim_{n \rightarrow \infty}\frac{1}{t_n}\int_0^{t_n} 
(R_{t+s}\psi)(x_0)\,{\rm d}t = 
\int_{\mathcal{A}}R_s\psi(x)\, \mu(\d x) \qquad \forall \, \psi \in C_b(\mathcal{A}).
\end{equation}
Now, bearing in mind equality \eqref{lim0} we have
\begin{align*}
&\frac{1}{t_n}\int_0^{t_n} (R_{t+s}\psi)(x_0)\,{\rm d}t 
= \frac{1}{t_n}\int_s^{s+t_n} (R_{t}\psi)(x_0)\,{\rm d}t 
\\
&\qquad= 
\frac{1}{t_n}\int_0^{t_n} (R_{t}\psi)(x_0)\,{\rm d}t + 
\frac{1}{t_n}\int_{t_n}^{s+t_n} (R_{t}\psi)(x_0)\,{\rm d}t -  
\frac{1}{t_n}\int_0^{s} (R_{t}\psi)(x_0)\,{\rm d}t \\
&\qquad \longrightarrow \int_{\mathcal{A}}\psi(x)\,\mu(\d x) \qquad \text{as} \ t_n \rightarrow \infty.
\end{align*}
Taking this into account in the left had side of \eqref{lim} shows that $\mu$ is invariant.
\end{proof}

We are now ready to show that the transition semigroup $P$
of equation \eqref{eq_AC} admits invariant measures.
\begin{theorem}
\label{th:mis-inv-2d}
Assume {\bf H1--H2}. Then, the transition semigroup $P$ is Feller and
admits at least one invariant measure.
\end{theorem}
\begin{proof}
The result is a consequence of the Krylov-Bougoliuov Theorem~\ref{KriBou}, provided 
that we check that $P$ is Feller and the tightness property.\\
(i). Let us show first that $P$ is Feller: this follows directly 
the continuous dependence of the solution on the initial data. Indeed, 
let $t>0$ and $\varphi \in C_b(\mathcal{A})$ be fixed.
We have to prove that, given a sequence $(x_n )_n \subset \mathcal{A}$
which converges in $\mathcal{A}$ to
$x\in \A$ as $n \rightarrow \infty$,  the sequence $P_t\varphi(x_n)$ converges to
$P_t\varphi(x) $ as $n \rightarrow \infty$.
As a consequence of the continuous dependence property w.r.t.~the initial datum \eqref{difference}, we have that
\begin{equation*}
\|u(t;x_n)-u(t;x)\|_{L^2(\Omega; H)} \leq 
\|u^{x_n}-u^x\|_{L^2(\Omega; C([0,t];H))} \le C_t\|x_n-x\|_{H}
\end{equation*}
It follows that, as $n\to\infty$, $u(t;x_n)\rightarrow u(t;x)$ in $L^2(\Omega;H)$, 
hence also in probability.
This in turn implies that $\varphi(u(t;x_n))\rightarrow \varphi(u(t;x))$ in probability
by the continuity of $\varphi$. The boundedness of $\varphi$ and the Vitali Theorem 
yield in particular that $\varphi(u(t;x_n))\rightarrow \varphi(u(t;x))$ in $L^1(\Omega)$, and thus 
\begin{equation*}
|(P_t\varphi)(x_n)-(P_t \varphi)(x)| \le \mathbb{E}\left[\left |\varphi(u(t;x_n))- \varphi(u(t;x))\right| \right] \rightarrow 0,
\end{equation*}
as $n \rightarrow \infty$. This shows that $P$ is Feller.\\
(ii). We prove now that $P$ satisfies the tightness property of Theorem~\ref{KriBou}.
To this end, let $x_0=0 \in \mathcal{A}$ and let $u^{x_0}$ be the corresponding variational solution of problem \eqref{eq_AC}. We are going to show that the family of measures 
$(\mu_t)_{t>0}\subset \cP(\mathcal{A})$ defined by 
\begin{equation*}
\mu_t : A \mapsto \frac 1t \int_0^t (P_s\pmb{1}_{A})(0)\, {\rm d}s
=\frac 1t \int_0^t \mathbb{P}\left\{u(t;0) \in A \right\}\, {\rm d}s, \qquad A \in \cB(\mathcal{A}), \ t>0,
\end{equation*}
is tight.
Let $B_n$ be the closed ball in $V$ of radius $n \in \mathbb{N}$, and set $\bar B_n:= B_n \cap \mathcal{A}$.
Then $\bar B_n$ is a compact subset of $\mathcal{A}$ since the embedding $V \hookrightarrow H$ is compact. 
Hence, Lemma~\ref{tight} and the Chebychev inequality yield, for any $t>0$,
\begin{align*}
\mu_t(\bar B_n^C)
&=\frac 1t \int_0^t (P_s\pmb{1}_{\bar B_n^C})(0)\, {\rm d}s=
\frac 1t \int_0^t \mathbb{P}\left\{\|u(s;0)\|^2_V \ge n^2 \right\}\, {\rm d}s
\\
&\le \frac{1}{tn^2}\int_0^t \mathbb{E} \|u(s;0)\|^2_V\, {\rm d}s
\lesssim_{C_1, C_B, |D|, \nu} \frac{1}{n^2},
\end{align*}
from which
\begin{equation*}
\sup_{t>0}\mu_t(B_n^C) 
\lesssim_{C_1, C_B, |D|, \nu} \frac{1}{n^2}\rightarrow 0 \qquad \text{as} \ n\rightarrow \infty,
\end{equation*}
and the thesis follows.
\end{proof}

\subsection{Support of the invariant measures}
Once existence of invariant measures is establishes, we focus 
here on some qualitative properties of the invariant measures concerning their support.
In particular, we show that every invariant measure is supported in 
a more regular set than just $\A$. To this end, 
we introduce the set 
\begin{equation}
  \label{A_sr}
  \A_{str}:=\left\{x\in \A\cap Z: \ F'(x)\in H\right\}.
\end{equation}
Proceeding as for \eqref{mat_A} and exploiting the lower semicontinuity of $|F'|$, 
one can show that $\A_{str}$ is a Borel subset of $H$, hence of $\A$.

\begin{proposition}
\label{supp1}
Assume {\bf H1--H2}. Then, there exists a constant $C>0$,
only depending on $C_0$, $C_1$, $C_B$, $|D|$, $\nu$, $K$, and $K_0$, such that
every invariant measure $\mu \in \cP(\mathcal{A})$
for the transition semigroup $P$ satisfies 
\begin{equation}
\label{int_mu_V}
\int_{\mathcal{A}}\left(\|x\|_Z^2 + \|F'(x)\|_H^2\right)\, \mu({\rm d}x)\le C.
\end{equation}
In particular, every invariant measure $\mu$ is supported in $\A_{str}$, i.e.~$\mu(\A_{str})=1$.
\end{proposition}
\begin{proof}
Let $\mu \in \cP(\mathcal{A})$ be an invariant measure 
for the transition semigroup $P$.\\
{\sc Step 1.} First we note that the definition of $\A$ itself
trivially implies that $\mu$ has finite moments of any order on $H$. More precisely, 
it holds that 
\[
  \norm{x}_{L^\infty(D)} \le 1 \quad\forall\, x\in\A,
\]
which readily ensures by the embedding $L^\infty(D)\embed H$ that
\begin{equation}
\label{sti_H}
\int_{\mathcal{A}} \|x\|^2_H\, {\rm d}\mu (x) \le |D|.
\end{equation}
{\sc Step 2.} Now we show that 
\begin{equation}\label{sti_V}
\int_{\mathcal{A}} \|x\|^2_V\, \mu(\d x) \le C.
\end{equation}
To this end, we consider the mapping
$\Phi:\mathcal{A} \rightarrow [0,+\infty]$ defined as
\begin{equation*}
\Phi: x \mapsto \|x\|^2_{V}\pmb{1}_{\mathcal{A} \cap V}(x)+ \infty \ \pmb{1}_{\mathcal{A} \cap V^C}(x), \quad x\in\A,
\end{equation*}
and its approximations $\{\Phi_n\}_{n\in\enne}$, where for every $n\in\enne$
$\Phi_n:\A\to[0,n^2]$ is defined (setting $B_n^V$ as the closed ball of radius $n$ in $V$) as
\begin{equation*}
\Phi_n:x \mapsto 
\begin{cases}
\|x\|^2_{V} & \text{if} \ x \in B_n^V \cap \mathcal{A},
\\
n^2 & \text{otherwise},
\end{cases}
\quad x\in\A.
\end{equation*}
It is not difficult to check that actually $\Phi_n\in\cB_b(\A)$ for every $n\in\enne$.
Hence, exploiting the invariance of $\mu$, the boundedness of $\Phi_n$, the definition \eqref{P_t} of $P$,
and the Fubini-Tonelli Theorem we have that
\begin{align*}
&\int_{\mathcal{A}} \Phi_n(x)\, \mu(\d x)= 
\int_0^1 \int_{\mathcal{A}} \Phi_n(x) \, \mu(\d x)\, {\rm d}s=
\int_0^1 \int_{\mathcal{A}} P_s \Phi_n(x)\, \mu(\d x)\, {\rm d}s\\
&\qquad= \int_0^1 \int_{\mathcal{A}} \E\left[\Phi_n(u(s;x))\right]\, \mu(\d x)\, {\rm d}s
=\int_{\mathcal{A}}\int_0^1 \mathbb{E}\left[\Phi_n(u(s;x)) \right]\,{\rm d}s\, \mu(\d x).
\end{align*}
Since 
\begin{equation*}
\Phi_n(\cdot)= \|\cdot\|^2_{V}  \wedge n^2 \le \|\cdot\|^2_{V},
\end{equation*}
by Lemma~\ref{tight} and \eqref{sti_H} we infer that 
\begin{align*}
  \int_{\mathcal{A}} \Phi_n(x)\, \mu(\d x) &\le
  \int_{\mathcal{A}}\int_0^1 \mathbb{E}\left[\|u(s;x)\|_V^2 \right]\,{\rm d}s\, \mu(\d x)\\
  &\lesssim \int_{\mathcal{A}}\|x\|_H^2\,\mu(\d x) + 1 \leq C.
\end{align*}
Since $\Phi_n$ converges pointwise and monotonically from below to $\Phi$, 
the Monotone Convergence Theorem yields \eqref{sti_V}.\\
{\sc Step 3.} We prove now that 
\begin{equation}\label{sti_Z}
\int_{\mathcal{A}} \|x\|^2_Z\, \mu(\d x) \le C.
\end{equation}
To this end, we argue as in {\sc Step 2}, considering the map
$\Psi:\mathcal{A} \rightarrow [0,+\infty]$ defined as
\begin{equation*}
\Psi: x \mapsto \|x\|^2_{Z}\pmb{1}_{\mathcal{A} \cap Z}(x)+ \infty \ \pmb{1}_{\mathcal{A} \cap Z^C}(x), \quad x\in\A,
\end{equation*}
and its approximations $\{\Psi_n\}_{n\in\enne}$, where for every $n\in\enne$
$\Psi_n:\A\to[0,n^2]$ is defined (setting $B_n^Z$ as the closed ball of radius $n$ in $Z$) as
\begin{equation*}
\Psi_n:x \mapsto 
\begin{cases}
\|x\|^2_{Z} & \text{if} \ x \in B_n^Z \cap \mathcal{A},
\\
n^2 & \text{otherwise},
\end{cases}
\quad x\in\A.
\end{equation*}
Again, one has that $\Psi_n\in\cB_b(\A)$ for every $n\in\enne$.
Hence, arguing as above by using 
the invariance of $\mu$, the boundedness of $\Psi_n$, the definition of $P$,
and the Fubini-Tonelli Theorem, exploiting the fact that 
$\mu$ is concentrated on $\A\cap V$ by \eqref{sti_V} yields
\begin{align*}
\int_{\mathcal{A}} \Psi_n(x)\, \mu(\d x)=
\int_{\mathcal{A}\cap V} \Psi_n(x)\, \mu(\d x)
&=\int_{\mathcal{A}\cap V}\int_0^1 \mathbb{E}\left[\Psi_n(u(s;x)) \right]\,{\rm d}s\, \mu(\d x)\\
&\le \int_{\mathcal{A}\cap V}\int_0^1 \mathbb{E}\left[\|u(s;x)\|_Z^2 \right]\,{\rm d}s\, \mu(\d x).
\end{align*}
Lemma~\ref{H^2_est} together with the estimate \eqref{sti_V} entail then
\[
  \int_{\mathcal{A}} \Psi_n(x)\, \mu(\d x) \lesssim
  \int_{\mathcal{A}}\|x\|_V^2\,\mu(\d x) + 1 \leq C.
\]
The Monotone Convergence Theorem establish then \eqref{sti_Z}.\\
{\sc Step 4.} Eventually, we show here that 
\begin{equation}\label{sti_F'}
\int_{\mathcal{A}} \|F'(x)\|_H^2\, \mu(\d x) \le C
\end{equation}
by arguing as above.
Define $\Lambda:\mathcal{A} \rightarrow [0,+\infty]$ as
\begin{equation*}
\Lambda: x \mapsto 
\begin{cases}
  \|F'(x)\|^2_{H} \quad&\text{if } F'(x) \in H,\\
  +\infty \quad&\text{otherwise},
\end{cases}
\quad x\in \A,
\end{equation*}
and its approximations $\{\Lambda_n\}_{n\in\enne}$, where for every $n\in\enne$
$\Lambda_n:\A\to[0,n^2]$ is defined as
\begin{equation*}
\Lambda_n:x \mapsto 
\begin{cases}
\|F'(x)\|^2_{H} & \text{if} \ \|F'(x)\|_H\le n,\\
n^2 & \text{otherwise},
\end{cases}
\quad x\in\A.
\end{equation*}
As above, it holds that $\Lambda_n\in\cB_b(\A)$ for every $n\in\enne$.
Using the
the invariance of $\mu$, the boundedness of $\Lambda_n$, the definition of $P$,
the Fubini-Tonelli Theorem, and the fact that 
$\mu$ is concentrated on $\A\cap V$, we infer that
\begin{align*}
\int_{\mathcal{A}} \Lambda_n(x)\, \mu(\d x)=
\int_{\mathcal{A}\cap V} \Lambda_n(x)\, \mu(\d x)
&=\int_{\mathcal{A}\cap V}\int_0^1 \mathbb{E}\left[\Lambda_n(u(s;x)) \right]\,{\rm d}s\, \mu(\d x)\\
&\le \int_{\mathcal{A}\cap V}\int_0^1 \mathbb{E}\left[\|F'(u(s;x))\|_H^2 \right]\,{\rm d}s\, \mu(\d x).
\end{align*}
At this point, Lemma~\ref{lem:F'} implies directly that
\[
  \int_{\mathcal{A}} \Lambda_n(x)\, \mu(\d x)  \leq C,
\]
and \eqref{sti_F'} follows from 
the Monotone Convergence Theorem.
\end{proof}

\subsection{Existence of an ergodic invariant measure}
Let us recall first the definition of ergodicity for the transition semigroup $P$.
In this direction, note that for every invariant measure $\mu$,
by density and by definition of invariance
the semigroup $P$ can be extended (with the same symbol for brevity) to 
a strongly continuous linear semigroup of contractions on $L^p(\A,\mu)$ for every $p\in[1,+\infty)$.

\begin{definition}
  \label{def:erg}
  An invariant measure $\mu\in\cP(\A)$ for the semigroup $P$ is said to be ergodic if
  \[
  \lim_{t \rightarrow \infty}\frac 1t \int_0^tP_s \varphi\, {\rm d}s 
  = \int_{\mathcal{A}}\varphi(x) \, \mu(\d x) \qquad \text{in} \ L^2(\mathcal{A}, \mu)\quad \forall\, \varphi \in L^2(\mathcal{A}, \mu).
  \]
\end{definition}

The estimate \eqref{int_mu_V} implies that the set of ergodic invariant measures is not empty.
More precisely, we have the following result.
\begin{proposition}
\label{ergodic}
Assume {\bf H1--H2}. Then,
there exists an ergodic invariant measure for the transition semigroup $P$.
\end{proposition}
\begin{proof}
It is well known (see e.g.\cite[Theorem 19.25]{Ali}) that for an arbitrary Markov transition semigroup $\{P_t\}_{t\ge 0}$, the ergodic measures are precisely the extreme points of the (possibly empty) convex set of its invariant measures. On the other hand, the Krein-Milman Theorem (see e.g.\cite[ Theorem 7.68]{Ali}) characterizes the convex compact sets, in locally convex Hausdorff spaces, as closed convex hull of its extreme points.
Let us denote by $\Pi\subset \cP(\mathcal{A})$ the convex set of all invariant measures for the Markov semigroup $\{P_t\}_{t\ge 0}$. In Theorem \ref{th:mis-inv-2d} we proved that $\Pi$ is non empty, thus, in view of the above discussion, it only remains to show that its closure is compact or equivalently that $\Pi$ is tight. 
By estimate \eqref{int_mu_V} in Proposition \ref{supp1} we know that there exists a constant $C$, depending on
the structural data, such that
\begin{equation*}
\int_{\mathcal{A}}\|x\|^2_V\, \mu( {\rm d}x) \le C \qquad \forall\, \mu \in \Pi.
\end{equation*}
Therefore, using the same notation of the proof of Theorem~\ref{th:mis-inv-2d}, by the Markov inequality we infer
that
\begin{align*}
\sup_{\mu \in \Pi} \mu \left(\bar B_n^C\right)
=\sup_{\mu \in \Pi} \mu (\{ x \in \mathcal{A} \ : \ \|x\|_V >n\}) \le \frac{1}{n^2}\sup_{\mu \in \Pi} \int_{\mathcal{A}}\|x\|^2_V\, \mu({\rm d}x)\le \frac{C}{n^2} \rightarrow 0,
\end{align*}
as $ n \rightarrow \infty$. Hence $\Pi$ is tight and admits extreme points,
which are ergodic invariant measures for $P$.
\end{proof}

\subsection{Uniqueness of the invariant measure}
Intuitively, uniqueness of invariant measures depends on how 
dissipative the stochastic equation \eqref{eq_AC} really is.
Here, we show that for a ``large enough'' 
diffusion coefficient $\nu$, the invariant measure is unique
and strongly mixing, according to the following definition.
\begin{definition}
  \label{def:st_mix}
  An invariant measure $\mu\in\cP(\A)$ for the semigroup $P$ is said to be strongly mixing if
  \[
  \lim_{t \rightarrow \infty}P_t \varphi
  = \int_{\mathcal{A}}\varphi(x) \, \mu(\d x) \qquad 
  \text{in} \ L^2(\mathcal{A}, \mu)\quad \forall\, \varphi \in L^2(\mathcal{A}, \mu).
  \]
\end{definition}

\begin{theorem}
\label{uniq_thm}
Assume {\bf H1--H2} and suppose that
\begin{equation}
\label{cond_uniq_inv}
\alpha_0:=\nu\left(\frac{1}{K_0^2}-1\right)-\frac{C_B}{2}-K >0.
\end{equation}
Then, there exists a unique invariant measure $\mu$ for the transition semigroup $P$.
 Moreover, $\mu$ is ergodic and strongly mixing.
\end{theorem}
\begin{remark}
Note that condition \eqref{cond_uniq_inv} is relevant since $K_0\in(0,1)$ by definition of $K_0$ itself.
Roughly speaking, the dissipativity inequality \eqref{cond_uniq_inv} is satisfied
either when the diffusion coefficient $\nu$ is large enough 
or when the structural coefficient $K_0$ is small enough.
For the latter case, we recall that $K_0$ depends exclusively on the domain $D$,
and can be estimated in terms of 
the first positive eigenvalue of the Laplacian operator on $D$ (according to the boundary conditions).
\end{remark}
\begin{proof}[Proof of Theorem~\ref{uniq_thm}]
Let $x,y\in \mathcal{A}$ and let $u^x,u^y$ be the respective variational solutions to problem \eqref{eq_AC}. 
Setting $w:=u^x-u^y$, the It\^o formula for the square of the $H$ norm of $w$
yields  for every $t \ge 0$, $\mathbb{P}$-almost surely, that
\begin{align}
\label{uniq1}
&\frac 12 \|w(t)\|^2_H + \nu \int_0^t \|\nabla w(s)\|^2_H\, {\rm d}s+ 
\int_0^t (F'(u^x(s))-F'(u^y(s)),w(s))_H\, {\rm d}s
\notag\\
&=\frac 12 \|x-y\|^2_H+\frac 12 \int_0^t \|B(u^x(s))-B(u^y(s))\|^2_{\mathcal L_{HS}(U,H)}\, {\rm d}s
\notag\\
& \qquad  + \int_0^t (w(s), \left(B(u^x(s))-B(u^y(s))\right){\rm d}W(s))_H.
\end{align} 
Using the Lipschitz continuity of the operator $B$ in {\bf H2} we estimate
\begin{equation*}
\frac 12 \int_0^t \|B(u(s))-B(v(s))\|^2_{\mathcal L_{HS}(U,H)}\, {\rm d}s
\le \frac {C_B}{2} \int_0^t \|w(s)\|^2_{H}\, {\rm d} s,
\end{equation*}
while exploiting assumption {\bf H1} we have that 
\[
  \int_0^t (F'(u^x(s))-F'(u^y(s)),w(s))_H\, {\rm d}s \ge -K\int_0^t \|w(s)\|^2_{H}\, {\rm d} s.
\]
Noting that the last term in \eqref{uniq1} is a square integrable martingale
thanks to \eqref{HS_norm} and the regularity of $w$, taking expectations we infer that 
\begin{align*}
&\frac 12 \E\|w(t)\|^2_H + \nu \E\int_0^t \|\nabla w(s)\|^2_H\, {\rm d}s
\le \frac 12 \|x-y\|^2_H+\left(\frac{C_B}{2} + K\right)
\E \int_0^t \|w(s)\|^2_H\, {\rm d}s.
\end{align*} 
It follows that
\begin{align*}
&\frac 12 \E\|w(t)\|^2_H + \nu \E\int_0^t \|w(s)\|^2_V\, {\rm d}s\\
&\qquad\le \frac 12 \|x-y\|^2_H+\left(\frac{C_B}{2} + K + \nu\right)
\E \int_0^t \|w(s)\|^2_H\, {\rm d}s,
\end{align*} 
hence also, thanks to the continuous inclusion $V\embed H$, that
\begin{align*}
&\frac 12 \E\|w(t)\|^2_H + \alpha_0\E\int_0^t \|w(s)\|^2_H\, {\rm d}s
\le \frac 12 \|x-y\|^2_H.
\end{align*} 
Now, exploiting the fact that 
$x, y \in \mathcal{A}$, hence in particular $\|x-y\|_{L^\infty} \le 2$,
by the Gronwall lemma we obtain 
\begin{equation}\label{uniq3}
  \E\|(u^x-u^y)(t)\|^2_H \le e^{-\alpha_0 t} \|x-y\|^2_H
  \le 4 |D| e^{-\alpha_0 t} \quad\forall\,t\ge0, \quad\forall\,x,y\in H.
\end{equation}
Consequently, let $\mu$ be an invariant measure for 
$\{P_t\}_{t\ge 0}$. For any $\varphi \in C_b^1(H)$ and 
$x \in \mathcal{A}$, by definition of invariance 
and the estimate \eqref{uniq3} we have
\begin{align*}
\left| P_t(\varphi_{|\A})(x)-\int_{\mathcal{A}}\varphi(y)\, \mu({\rm d}y)\right|^2 
&\le \|D\varphi\|^2_{\infty}\int_{\mathcal{A}}\mathbb{E}\left[ \|u^x(t)-u^y(t)\|^2_H\right]\, \mu({\rm d}y)\\
&\le 2|D|\|D\varphi\|^2_{\infty}e^{-\alpha_0 t}
\end{align*}
uniformly in $x$. Since $C^1_b(H)_{|\A}$ is dense in $L^2(\mathcal{A}, \mu)$, we deduce that 
\begin{equation*}
\left| P_t\varphi(x)-\int_{\mathcal{A}}\varphi(y)\, \mu({\rm d}y)\right|\rightarrow 0
 \quad \text{as} \ t \rightarrow \infty, \quad \forall \varphi \in L^2(\mathcal{A}, \mu),
\end{equation*}
that is the strong mixing property holds true. 
Notice that the above computation easily implies also the uniqueness of the invariant measure. 
Indeed, let $\pi$ be another invariant measure, then for all $\varphi \in C_b^1(H)$ we have
\begin{align*}
&\left|  \int_{\mathcal{A}}\varphi(y)\, \mu({\rm d}y)- \int_\mathcal{A}\varphi(x)\, \pi({\rm d}x) \right| \\
&=\left|  \int_{\mathcal{A} }\int_{\mathcal{A}}\left(P_t(\varphi_{|\A})(y)
-P_t(\varphi_{|\A})(x)\right)\,  \pi({\rm d}x)\mu({\rm d}y) \right|
\\
&\le 2|D|\|D\varphi\|^2_{\infty} e^{-\alpha_0 t} \rightarrow 0 \qquad \text{as} \  t \rightarrow \infty.
\end{align*}
The fact that the unique invariant measure is also ergodic follows from 
Proposition~\ref{ergodic}, and this concludes the proof.
\end{proof}

\section{Analysis of the Kolmogorov equation}
\label{sec:kolm}
In this section we focus on the Kolmogorov operator associated to the stochastic 
equation \eqref{eq_AC}. We aim at characterising the infinitesimal generator of the 
transition semigroup $P$ in terms of the closure of the Kolmogorov operator associated to 
\eqref{eq_AC} in the space $L^2(\A,\mu)$, where $\mu$ is an invariant measure for $P$.

Throughout the section, we assume {\bf H1--H2} and 
$\mu$ is an invariant measure for the semigroup $P$.
We have already pointed out that $P$ extends by density to a
strongly continuous linear semigroup 
of contractions on $L^2(\A,\mu)$, which 
will be denoted by the same symbol 
$P$ for convenience.
As such, for the semigroup $P$ on $L^2(\A,\mu)$ it is well defined the 
infinitesimal generator $(L,D(L))$, namely 
\[
  D(L):=\left\{\varphi\in L^2(\A,\mu): \quad
  \lim_{t\to0^+}\frac{P_t\varphi - \varphi}{t} \text{ exists in } L^2(\A,\mu)\right\}
\]
and
\[
  -L\varphi:=\lim_{t\to0^+}\frac{P_t\varphi - \varphi}{t} \quad\text{in } L^2(H,\mu), \quad \varphi\in D(L).
\]
The main issue that we address is to characterise the infinitesimal generator $(L,D(L))$
in terms of the  Kolmogorov operator associated to \eqref{eq_AC}.

\subsection{The Kolmogorov operator}
We define the Kolmogorov operator $(L_0,D(L_0))$ associated to the stochastic 
equation \eqref{eq_AC} as follows. 
We set 
\[
  D(L_0):=C^2_b(H)_{|\A}=\left\{\varphi\in C_b(\A):\exists\,\psi\in C^2_b(H): \varphi(x)=\psi(x) \quad\forall\,x\in\A\right\}.
\]
and 
\begin{align*}
  L_0\varphi(x)&:=-\frac12\operatorname{Tr}[B(x)^*D^2\varphi(x)B(x)] + (-\Delta x + F'(x), D\varphi(x))_H,\\
  &\qquad x\in \A_{str}, \quad \varphi\in D(L_0).
\end{align*}
Note that the definition of the domain of $L_0$ through restrictions on $\A$ is essential,
as the operators $B$ and $F'$ are not defined on the whole $H$. More specifically, let us stress that 
not even considering $x\in\A$ is enough: this is because $F'(x)$ makes sense in $H$ only for $x\in\A_{str}$,
and not for any $x\in\A$. 

We note that for every $\varphi\in D(L_0)$, with this definition 
the element $L_0\varphi$ is actually well defined as an element in $L^2(\A,\mu)$.
Indeed, thanks to the estimate \eqref{HS_norm} one has,
for every $y\in\A_{str}$, that
\begin{align*}
  |L_0\varphi(y)|&\le \norm{\varphi}_{C^2_b(H)}
  \left(\norm{B(y)}_{\mathcal L_{HS}(U,H)}^2 + \norm{F'(y)}_{H} + \norm{\Delta y}_H\right)\\
  &\le\norm{\varphi}_{C^2_b(H)}\left(C_B + \norm{F'(y)}_{H} + \norm{\Delta y}_H\right),
\end{align*}
so that the estimate \eqref{int_mu_V} yields that 
\[
  L_0\varphi \in L^2(\A,\mu) \quad\forall\,\varphi \in D(L_0).
\]
The fact that $L_0\varphi$ is explicitly defined only on $\A_{str}$, and not on $\A$, 
is irrelevant when working in $L^2(\A,\mu)$ since $\mu(\A_{str})=1$.
It follows then that $(D(L_0), L_0)$
is a linear unbounded operator on the Hilbert space $L^2(\A,\mu)$.

The elliptic Kolmogorov equation associated to \eqref{eq_AC} reads
\beq
  \label{eq:kolm}
  \alpha\varphi(x) + L_0\varphi(x) = g(x), \quad x\in\A_{str},
\eeq
where $\alpha>0$ is a given coefficient and $g:\A\to\erre$ is a given datum. 

The first natural result is the following.
\begin{lemma}
  \label{lem:L}
  In this setting, it holds that $D(L_0)\subset D(L)$ and
  \[
  L\varphi(x)=L_0\varphi(x) \quad\text{for $\mu$-a.e.~$x\in\A$,} \quad\forall\,\varphi\in D(L_0).
  \]
\end{lemma}
\begin{remark}
  As we have already point out above, 
  notice that in this identity the expression $L_0\varphi(x)$ makes 
  sense for $\mu$-almost every~$x\in\A$ (and not just in $\A_{str}$) by virtue of 
  Proposition~\ref{supp1}, which ensures indeed that $\mu(\A_{str})=1$.
\end{remark}
\begin{proof}[Proof of Lemma~\ref{lem:L}]
Let $x\in\A \cap V$ and let $u:=u^x$ be the  respective unique analytically strong
solution to \eqref{eq_AC}. Then, 
for every $\varphi\in D(L_0)$ the It\^o formula yields directly, for every $t\ge0$,
\begin{align*}
  &\E\varphi(u(t)) + \E\int_0^t(- \nu \Delta u(s) + F'(u(s)), D\varphi(u(s)))_H\,\d s\\
  &\qquad=\E\varphi(x) +\frac12 \E\int_0^t \operatorname{Tr}[B(u(s))^*D^2\varphi(u(s))B(u(s))]\,\d s.
\end{align*}
Since $u(t)\in\A$ $\P$-almost surely for every $t\ge0$, this yields 
\beq\label{ito_L}
  P_t\varphi(x) - \varphi(x) + \int_0^t P_s(L_0\varphi)(x)\,\d s = 0 \quad\forall\,x\in\A.
\eeq
Since $L_0\varphi \in L^2(\A,\mu)$ and $P$  is strongly continuous on $L^2(\A,\mu)$, this implies 
\[
  s\mapsto P_s(L_0\varphi) \in C([0,t]; L^2(\A,\mu)),
\]
from which it follows that 
\[
  \lim_{t\to0^+}\frac1t\int_0^t P_s(L_0\varphi)\,\d s = 
  L_0\varphi \quad\text{in } L^2(\A,\mu).
\]
This shows by comparison in \eqref{ito_L} that $\varphi\in D(L)$.
Moreover, dividing  by $t$ and taking the limit in $L^2(\A,\mu)$ as $t\to0^+$  in
the identity \eqref{ito_L}, we get that
\[
  -L\varphi(x) + L_0\varphi(x) = 0 \quad\text{for $\mu$-a.e.~$x\in\A$,}
\]
and we conclude.
\end{proof}

\subsection{A regularised Kolmogorov equation}\label{ssec:kolm_reg}
A first main issue that we aim at addressing is to investigate 
existence and uniqueness of strong solutions 
to the Kolmogorov equation \eqref{eq:kolm} in $L^2(\A,\mu)$.
However, some preliminary preparations are necessary.
In particular, we construct here a family of regularised Kolmogorov equations
that approximate \eqref{eq:kolm} and for which we are actually able 
to show existence of {\em classical} solutions on the whole space $H$.
This will be done by using a double approximation of the operators:
one in the parameter $\lambda>0$, which basically removes the singularity  of $F'$
and allows to work on $H$ rather than just $\A$, and one in the parameter $n\in\enne$,
which conveys enough smoothness to the operators themselves.

Due to the presence of the multiplicative noise, in order to tackle the 
Kolmogorov equation, in the current Subsection~\ref{ssec:kolm_reg}
we shall need the following reinforcement of assumption {\bf H2},
namely:
\begin{description}
  \item[H2'] The sequence $\{h_k\}_{k \in \mathbb{N}}$ is included also in $C^2([-1,1])$
  and satisfies for every $k \in \mathbb{N}$ that $h_k'(\pm1)=0$. Moreover, it holds that 
 \begin{equation}
  \label{C_B'}
  C_B' := \sum_{k \in \mathbb{N}}\|h_k''\|^2_{C([-1,1])}  < \infty.		
\end{equation}
\end{description}

\subsubsection{First approximation}
Thanks to the assumption {\bf H1}, 
the function
\begin{equation}\label{beta}
  \beta:(-1,1)\to\erre, \qquad\beta(r):=F'(r) + Kr,\quad r\in(-1,1),
\end{equation}
is continuous non-decreasing, hence can be identified to a maximal monotone graph in $\erre\times\erre$.
In particular, for every $\lambda>0$ it is well defined its resolvent operator 
$J_\lambda:=(I+\lambda\beta)^{-1}:\erre\to(-1,1)$, i.e.~for every $r\in\erre$,
$J_\lambda(r)$ is the unique element in $(-1,1)$ such that 
$J_\lambda(r) + \lambda\beta(J_\lambda(r))=r$. 
The Yosida approximation of $\beta$ is 
defined as $\beta_\lambda:\erre\to\erre$, $\beta_\lambda(r):=\beta(J_\lambda(r))$, $r\in\erre$.
We recall that $\beta_\lambda$ is Lipschitz-continuous and non-decreasing:
for further properties on monotone and convex analysis we refer to \cite{Barbu_NLDEMT}.
\\
Let also $\rho\in C^\infty_c(\erre)$ with $\operatorname{sup}(\rho)=[-1,1]$, $\rho\ge0$, $\|\rho\|_{L^1(\erre)}=1$,
and set
\[
  \rho_\lambda:\erre\to\erre, \qquad
  \rho_\lambda(r):=\lambda^{-1}\rho(\lambda^{-1} r), \quad r\in\erre, \qquad\lambda>0,
\]
so that $(\rho_\lambda)_{\lambda>0}$ is a usual sequence of mollifiers on $\erre$. Let us set for convenience 
\[ 
  c_\rho:=\norm{\rho'}_{L^1(\erre)}.
\]

Let us construct the approximated operators. First of all, for $\lambda>0$ we define
the $\lambda$-regularised potential 
$F_\lambda:\erre\to[0,+\infty)$
as
\[ 
  F_\lambda(x):=F(0)-\frac{K}2|x|^2+\int_0^x(\rho_{\lambda^2}\star\beta_\lambda)(y)\,\d y, \quad x\in\erre.
\]
Note that one has
\beq\label{def_app1}
  F_\lambda'(x)=(\rho_{\lambda^2}\star\beta_\lambda)(x) - Kx, \quad x\in\erre,
\eeq
and from the properties of the Yosida approximation and convolutions 
it is not difficult to show that 
\beq
  \label{rate_lambda}
  |F_\lambda''(x)|\leq K+\frac1\lambda \quad\forall\,x\in\erre, \qquad
  |F_\lambda'''(x)|\leq \frac{c_\rho}{\lambda^3} \quad\forall\,x\in\erre.
\eeq
For clarity, let us use a separate notation for the superposition 
operator induced by the Lipschitz real function $F_\lambda'$ on 
the Hilbert space $H$, namely we set
\[
  \mathcal F_\lambda:H\to H, \qquad 
  (\mathcal F_\lambda(v))({\bf x}):=F_\lambda'(v({\bf x})) \quad\text{for a.e.~${\bf x}\in D$,}
  \quad v\in H.
\]
Since $F_\lambda'\in C^\infty(\erre)$ by definition and has bounded derivatives of any order, 
thanks to the continuous embedding $V\embed L^6(D)$ and the dominated convergence theorem, 
$\mathcal F_\lambda$ can be shown to be twice Fr\'echet differentiable in $H$ along directions of $V$:
more precisely, this means that for every $x\in H$ there exist two operators 
\[
  D\mathcal F_\lambda(x) \in \mathcal L(V,H), \qquad
  D^2\mathcal F_\lambda(x) \in \mathcal L(V; \mathcal L(V,H))\cong\mathcal L_2(V\times V; H)
\]
such that 
\begin{align*}
  \lim_{\norm{h}_V\to0}
  &\frac{\norm{\mathcal F_\lambda(x+h)-\mathcal F_\lambda(x)-D\mathcal F_\lambda(x)[h]}_H}{\norm{h}_V}=0,\\
  \lim_{\norm{h}_V\to0}
  &\frac{\norm{D\mathcal F_\lambda(x+h)-D\mathcal F_\lambda(x)-
  D^2\mathcal F_\lambda(x)[h,\cdot]}_{\mathcal L(V,H)}}{\norm{h}_V}=0.
\end{align*}
In particular, one has that 
\begin{alignat*}{2}
  D\mathcal F_\lambda(x)[h]&=F_\lambda''(x)h, \qquad &&x\in H,\quad h\in V,\\
  D^2\mathcal F_\lambda(x)[h_1,h_2]&=F_\lambda'''(x)h_1h_2, \qquad &&x\in H,\quad h_1,h_2\in V,
\end{alignat*}
so that \eqref{rate_lambda} yields, for a constant $c>0$ only depending on $\rho$, $K$, and $D$,
\begin{alignat}{2}
\label{rate_lambda2}
  \norm{D\mathcal F_\lambda(x)}_{\mathcal L(V,H)}&\leq c\left(1+\frac1\lambda\right) \quad&&\forall\,x\in H,\\
\label{rate_lambda3}
  \norm{D^2\mathcal F_\lambda(x)}_{\mathcal L(V; \mathcal L(V,H))}&\leq 
  \frac{c}{\lambda^3} \quad&&\forall\,x\in H.
\end{alignat}

As far as the operator $B$ is concerned, 
for every $k\in\enne$
we first extend $h_k$ to $\tilde h_k:\erre\to\erre$ by setting 
\[
  \tilde h_k(x):=
  \begin{cases}
  h_k(x) \quad&\text{if } x\in[-1,1],\\
  0 &\text{otherwise},
  \end{cases}
  \quad k\in\enne.
\]
In this way, by assumptions {\bf H2} and {\bf H2'}
it is clear that $\{\tilde h_k\}_{k\in\enne}\subset W^{2,\infty}(\erre)$, and we define then 
\[
  h_{k,\lambda}:=\rho_{\lambda^\gamma}\star \tilde h_k, \quad k\in\enne, \quad \lambda>0,
\]
where $\gamma>0$ is a prescribed fixed rate coefficient that will be chosen later.
The reason of introducing $\gamma$ here may sound not intuitive at this level, and will be clarified 
in the following sections: roughly speaking, $\gamma$ is needed in order to suitable compensate for 
the blow-up in \eqref{rate_lambda3}.
With these definitions, we set then
\beq\label{def_app2}
  B_\lambda:H\to \mathcal L_{HS}(U,H), \qquad 
  B_\lambda(x)e_k:=h_{k,\lambda}(x), \quad x\in H, \quad k\in\enne.
\eeq
Clearly, for every $\lambda$ it holds that $\{h_{k,\lambda}\}_{k\in\enne}\subset C^\infty_c(\erre)$
and, by the properties of convolutions and assumptions {\bf H2} and {\bf H2'}, 
for every $\lambda>0$ it holds that
\beq\label{rate_lambda4}
  \sum_{k\in\enne}\norm{h_{k,\lambda}}_{C^2(\erre)}^2
  \leq \sum_{k\in\enne}\|\tilde h_{k}\|_{W^{2,\infty}(\erre)}^2
  =\sum_{k\in\enne}\norm{h_{k}}_{C^2([-1,1])}^2\leq C_B+C_B'.
\eeq
It follows in particular that for every $\lambda>0$ the operator $B_\lambda$ constructed above
is $\sqrt{C_B}$-Lipschitz continuous and bounded.
Also, similarly as above one can check that $B_\lambda$
is twice Fr\'echet differentiable along the directions of $V$, in the sense that  for every
$x\in H$ there exist two operators 
\begin{align*}
  D B_\lambda(x) &\in \mathcal L(V,\mathcal L_{HS}(U,H)), \\
  D^2 B_\lambda(x) &\in \mathcal L(V; \mathcal L(V,\mathcal L_{HS}(U,H)))
  \cong\mathcal L_2(V\times V; \mathcal L_{HS}(U,H))
\end{align*}
such that 
\begin{align*}
  \lim_{\norm{h}_V\to0}
  &\frac{\norm{ B_\lambda(x+h)- B_\lambda(x)-
  D B_\lambda(x)[h]}_{\mathcal L_{HS}(U,H)}}{\norm{h}_V}=0,\\
  \lim_{\norm{h}_V\to0}
  &\frac{\norm{D B_\lambda(x+h)-D B_\lambda(x)-
  D^2 B_\lambda(x)[h,\cdot]}_{\mathcal L(V,\mathcal L_{HS}(U,H))}}{\norm{h}_V}=0.
\end{align*}
More precisely, it holds that 
\begin{alignat*}{2}
  D B_\lambda(x)[z]e_k&=h_{k,\lambda}'((x))z, \qquad &&x\in H,\quad z\in V,\quad k\in\enne,\\
  D^2 B_\lambda(x)[z_1,z_2]e_k&=h_{k,\lambda}''((x))z_1z_2, \qquad &&x\in H,\quad z_1,z_2\in V,\quad k\in\enne.
\end{alignat*}
From the continuous embedding $V\embed L^6(D)$,
condition \eqref{rate_lambda4}
ensures the existence of a constant $c$ independent of $\lambda$ such that
\begin{alignat}{2}
\label{rate_lambda5}
  \norm{D B_\lambda(x)}_{\mathcal L(V,\mathcal L_{HS}(U,H))}^2&\leq C_B \quad&&\forall\,x\in H,\\
\label{rate_lambda6}
  \norm{D^2 B_\lambda(x)}_{\mathcal L(V, \mathcal L(V,\mathcal L_{HS}(U,H)))}^2&\leq 
  c\quad&&\forall\,x\in H.
\end{alignat}
Furthermore, it is not difficult to see that actually $B_\lambda$ is also G\^ateaux differentiable 
from the whole $H$ to $\mathcal L_{HS}(U,H)$, and $DB_{\lambda}$ is G\^ateaux differentiable 
along the directions of $L^4(D)$, so that for every $x\in H$ 
$DB_\lambda(x)$ and $D^2B_\lambda(x)$ extend to well defined operators in
the spaces $\mathcal L(H, \mathcal L_{HS}(U,H))$ and $\mathcal L(L^4(D),\mathcal L(L^4(D), \mathcal L_{HS}(U,H)))$,
respectively. Again by \eqref{rate_lambda4} we also have then 
\begin{alignat}{2}
\label{rate_lambda5'}
  \norm{D B_\lambda(x)}_{\mathcal L(H,\mathcal L_{HS}(U,H))}^2&\leq C_B \quad&&\forall\,x\in H,\\
\label{rate_lambda6'}
  \norm{D^2 B_\lambda(x)}_{\mathcal L(L^4(D), \mathcal L(L^4(D),\mathcal L_{HS}(U,H)))}^2&\leq 
  c\quad&&\forall\,x\in H.
\end{alignat}

\subsubsection{Second approximation}
While the approximation in $\lambda$ is enough for proving well posedness of the stochastic equation \eqref{eq_AC},
in order to approximate the Kolmogorov equation \eqref{eq:kolm} we need more smoothness on the coefficients.
In this direction, we shall rely on some smoothing operators in infinite dimensions.
Let now $\lambda>0$ be fixed.

Let $\mathcal C$ be the unbounded linear operator $(I-\Delta)$ on $H$ with effective domain $Z$
(note that the definition of $Z$ includes either Dirichlet or Neumann boundary conditions, according 
to $\alpha_d$ and $\alpha_n$).
Then, $\mathcal C$ is linear maximal monotone, coercive on $V$, and 
$-\mathcal C$ generates a strongly continuous semigroup of 
contractions $(e^{-t\mathcal C})_{t\ge0}$
on $H$. 
Furthermore, since we are working in dimension $d=2,3$,
it is possible to show that $\mathcal C^{-1}\in\mathcal L_{HS}(H,H)$:
this follows from the fact that the eigenvalues $\{\lambda_k\}_k$ of $\mathcal C$
satisfy $\lambda_k\approx 1+k^{2/d}$ as $k\to\infty$.
We consider the Ornstein-Uhlenbeck transition semigroup
$R:=(R_t)_{t\ge0}$ given by
\[
R_t\varphi(x):=\int_H\varphi(e^{-t\mathcal C}x+y)\,N_{Q_t}(\d y), \quad \varphi\in\cB_b(H),
\]
where
\beq\label{Qt}
  Q_t:=\int_0^te^{-s\mathcal C}\mathcal C^{-2}
  e^{-s\mathcal C}\,\d s = \frac12\mathcal C^{-3}\left(I - e^{-2t\mathcal C}\right)
  , \quad t\ge0.
\eeq
Note that $Q_t$ is trace class on $H$ for every $t\ge0$ and 
\beq
  \label{est_Qt}
  \operatorname{Tr}(Q_t) \leq \norm{\mathcal C^{-1}}_{\mathcal L_{HS}(U,H)}^2 t \quad\forall\,t\ge0.
\eeq
Moreover, 
if $\{c_k\}_k$ is a complete orthonormal system of $H$ made of eigenfunctions of $\mathcal C$,
with eigenvalues $\{\lambda_k\}_k$, one has
\begin{align*}
  \norm{Q_t^{-1/2}e^{-t\mathcal C}c_k}_H
  =\sqrt2\lambda_k^{3/2}\frac{e^{-\lambda_kt}}{(1-e^{-2\lambda_kt})^{1/2}}
  \leq \frac{\sqrt2}{t^{3/2}}\max_{r>0}\frac{r^{3/2}e^{-r}}{(1-e^{-2r})^{1/2}},
\end{align*}
from which it follows,
thanks to the characterisation of null-controllability in \cite[Prop.~B.2.1]{DapZab2}, that 
\beq
  \label{st_fell}
  e^{-t\mathcal C}(H)\subset Q_t^{1/2}(H) \qquad\text{and}\qquad
  \norm{Q_t^{-1/2}e^{-t\mathcal C}}_{\mathcal L(H,H)}\leq\frac{C}{t^{3/2}}, \quad \forall\,t>0.
\eeq
Thanks to \eqref{st_fell},
it is well known (see \cite[Thm.~6.2.2]{DapZab2}) that 
$R$ is strong Feller, in the sense that for every $\varphi\in \cB(H)$
and $t>0$ it holds that $R_t\varphi \in UC^\infty_b(H)$, as well as
$R_t\varphi(x)\to\varphi(x)$ for every $x\in H$ as $t\to0^+$.
It is natural then to introduce, for every $n\in\enne$, the regularisations
\[
  \mathcal F_{\lambda,n}:H\to H, \qquad
  B_{\lambda,n}:H\to \mathcal L_{HS}(U,H),
\]
as
\begin{align}
  \label{def_app3}
  \mathcal F_{\lambda,n}(x)&:=
  \int_H e^{-\frac{\mathcal C}n}\mathcal F_\lambda(e^{-\frac{\mathcal C}n}x+y)\,N_{Q_{1/n}}(\d y), \quad x\in H,\\
  \label{def_app4}
  B_{\lambda,n}(x)&:=\int_H e^{-\frac{\mathcal C}{n^\delta}}
  B_\lambda(e^{-\frac{\mathcal C}{n^\delta}}x+y)\,N_{Q_{1/n^\delta}}(\d y), \quad x\in H,
\end{align}
where $\delta>0$ is a positive rate coefficient that will be chosen later on.
Again, as for the case of $\gamma$ in the $\lambda$-approximation, 
the need of allowing for a general rate $\delta$ will be needed to 
suitably compensate the blow-up of $\mathcal F_{\lambda,n}$.

It follows for every $n\in\enne$ that the approximated operators satisfy 
$\mathcal F_{\lambda,n}\in C^\infty(H;H)$,
$B_{\lambda,n}\in C^\infty(H;\mathcal L_{HS}(U,H))$ and have bounded derivatives of any order. 
Moreover, since $e^{-\mathcal C/n}(H)\subset Z\subset V$ and 
$\mathcal F_\lambda$ and $B_\lambda$ are twice differentiable along directions of $V$,
for every $n\in\enne$ and for every $x,z,z_1,z_2\in H$ it holds that  
\begin{align*}
  D\mathcal F_{\lambda,n}(x)[z]&:=\int_H e^{-\frac{\mathcal C}n}
  D\mathcal F_\lambda(e^{-\frac{\mathcal C}{n}}x+y)[e^{-\frac{\mathcal C}n}z]\,N_{Q_{1/n}}(\d y),\\
  DB_{\lambda,n}(x)[z]&:=\int_H e^{-\frac{\mathcal C}{n^\delta}}
  DB_\lambda(e^{-\frac{\mathcal C}{n^\delta}}x+y)[e^{-\frac{\mathcal C}{n^\delta}}z]\,N_{Q_{1/n^\delta}}(\d y), 
\end{align*}
and
\begin{align*}
  D^2\mathcal F_{\lambda,n}(x)[z_1,z_2]&:=\int_H e^{-\frac{\mathcal C}{n}}
  D^2\mathcal F_\lambda(e^{-\frac{\mathcal C}{n}}x+y)
  [e^{-\frac{\mathcal C}{n}}z_1,e^{-\frac{\mathcal C}{n}}z_2]\,N_{Q_{1/n}}(\d y),\\
  D^2B_{\lambda,n}(x)[z_1,z_2]&:=\int_H e^{-\frac{\mathcal C}{n^\delta}}
  D^2B_\lambda(e^{-\frac{\mathcal C}{n^\delta}}x+y)
  [e^{-\frac{\mathcal C}{n^\delta}}z_1,e^{-\frac{\mathcal C}{n^\delta}}z_2]\,N_{Q_{1/n^\delta}}(\d y).
\end{align*}
In particular, from assumption {\bf H1}, the non-expansivity of $e^{-\mathcal C/n}$
and the estimates \eqref{rate_lambda5'}--\eqref{rate_lambda6'} we have that 
\begin{align}
\label{est1_kolm}
  (D\mathcal F_{\lambda,n}(x)[z],z)_H\ge -K\norm{z}_H^2 \qquad&\forall\,x,z\in H,\\
\label{est2_kolm}
  \norm{DB_{\lambda,n}(x)}^2_{\mathcal L(H,\mathcal L_{HS}(U,H))}\leq C_B \qquad&\forall\,x\in H,\\
\label{est3_kolm}
  \norm{D^2B_{\lambda,n}(x)}^2_{\mathcal L(L^4(D),\mathcal L(L^4(D),\mathcal L_{HS}(U,H)))}
  \leq c \qquad&\forall\,x\in H,
\end{align}
where we note that all constants $K$, $C_B$, and $c$ are independent of $\lambda$ and $n$.
Furthermore, exploiting the estimate \eqref{rate_lambda3} and the fact that $V=D(\mathcal C^{1/2})$,
for every $x,z_1,z_2\in H$ we infer that 
\begin{align*}
  &\norm{D^2\mathcal F_{\lambda,n}(x)[z_1,z_2]}_H\\
  &\qquad\leq \int_H\norm{D^2\mathcal F_\lambda(e^{-\frac{\mathcal C}n}x+y)}_{\mathcal L(V, \mathcal L(V,H))}
  \norm{e^{-\frac{\mathcal C}n}z_1}_V\norm{e^{-\frac{\mathcal C}n}z_1}_V\,N_{Q_{1/n}}(\d y)\\
  &\qquad\leq \frac{c}{\lambda^3}\norm{e^{-\frac{\mathcal C}n}z_1}_V\norm{e^{-\frac{\mathcal C}n}z_1}_V\\
  &\qquad\lesssim\frac{1}{\lambda^3} n^{1/2}\norm{z_1}_H n^{1/2}\norm{z_2}_H.
\end{align*}
It follows that there exists a positive constant $c>0$, independent of $n$ and $\lambda$, such that 
\beq
  \label{est4_kolm}
  \norm{D^2\mathcal F_{\lambda,n}(x)}_{\mathcal L(H,\mathcal L(H,H))}
  \leq c \frac{n}{\lambda^3}\qquad\forall\,x\in H.
\eeq
Analogously, thanks to the continuous embedding $H^{3/4}(D)\embed L^4(D)$ in dimensions $d=2,3$,
one has that $D(\mathcal C^{3/8})\embed L^4(D)$:
hence, proceeding as above and using \eqref{rate_lambda6'} instead,
one gets for every $x,z_1,z_2\in H$ that 
\begin{align*}
  &\norm{D^2 B_{\lambda,n}(x)[z_1,z_2]}_H\\
  &\leq \int_H\norm{D^2B_\lambda
  (e^{-\frac{\mathcal C}{n^\delta}}x+y)}_{\mathcal L(L^4(D), \mathcal L(L^4(D),\mathcal L_{HS}(U,H)))}\\
  &\qquad\times\norm{e^{-\frac{\mathcal C}{n^\delta}}z_1}_{L^4(D)}
  \norm{e^{-\frac{\mathcal C}{n^\delta}}z_1}_{L^4(D)}\,N_{Q_{1/n^\delta}}(\d y)\\
  &\leq c\norm{e^{-\frac{\mathcal C}{n^\delta}}z_1}_{L^4(D)}
  \norm{e^{-\frac{\mathcal C}{n^\delta}}z_1}_{L^4(D)}\\
  &\lesssim n^{\frac{3}{8}\delta}\norm{z_1}_H n^{\frac{3}{8}\delta}\norm{z_2}_H.
\end{align*}
It follows that there exists a positive constant $c>0$, independent of $n$ and $\lambda$, such that 
\beq
  \label{est5_kolm}
  \norm{D^2 B_{\lambda,n}(x)}_{\mathcal L(H,\mathcal L(H,\mathcal L_{HS}(U,H)))}
  \leq c n^{\frac34\delta}\qquad\forall\,x\in H.
\eeq

\subsubsection{Construction of classical solutions}
Let now $\lambda>0$ and $n\in\enne$ be fixed.
For every $x\in H$ the doubly approximated stochastic equation 
\beq\label{eq_app2}
\begin{cases}
{\rm d}u_{\lambda,n}(t)-\nu\Delta u_{\lambda,n}(t)\,{\rm d}t +
\mathcal F_{\lambda,n}(u_{\lambda,n}(t))\,{\rm d}t
\quad& \\
\quad= 
B_{\lambda,n}(u_{\lambda,n}(t))\,{\rm d}W(t) & \text{in} \ (0,T) \times D,
\\
\alpha_du_{\lambda,n} + \alpha_n\partial_{\bf n}u_{\lambda,n}=0 & \text{in} \ (0,T) \times \Gamma,
\\
u_{\lambda,n}(0)=x & \text{in} \ D,
\end{cases}
\eeq
admits a unique solution $u_{\lambda,n}=u_{\lambda,n}^{x}\in 
L^p(\Omega; C([0,T]; H)\cap L^2(0,T; V))$, for all $p\ge2$
and $T>0$.
Moreover, due to the smoothness of the coefficients $\mathcal F_{\lambda,n}$ and $B_{\lambda,n}$, 
by the regular dependence results in \cite{MPR}
(see also \cite{MS_fre}) one can infer in particular that the solution map satisfies, for all $T>0$,
\beq
  \label{reg_dep}
  S_{\lambda,n}:x\mapsto u_{\lambda,n}^x \in C^2_b(H; L^2(\Omega; C([0,T]; H))).
\eeq
Furthermore, for $x,z,z_1,z_2\in H$, the derivatives of $S_{\lambda,n}$ are given by 
\[
  DS_{\lambda,n}(x)[z]=v_{\lambda,n}^z, \qquad D^2S_{\lambda,n}(x)[z_1,z_2]=w_{\lambda,n}^{z_1,z_2},
\]
where
\[
  v_{\lambda,n}^z, w_{\lambda,n}^{z_1,z_2} \in L^2(\Omega; C([0,T]; H)\cap L^2(0,T; V)) \quad\forall\,T>0
\]
are the unique solutions the stochastic equations 
\beq\label{eq_app_v}
\begin{cases}
{\rm d}v^z_{\lambda,n}(t)-\nu\Delta v^z_{\lambda,n}(t)\,{\rm d}t +
D\mathcal F_{\lambda,n}(u_{\lambda,n}(t))v^z_{\lambda,n}(t)\,{\rm d}t \quad& \\
\quad= 
DB_{\lambda,n}(u_{\lambda,n}(t))v^z_{\lambda,n}(t)\,{\rm d}W(t) & \text{in} \ (0,T) \times D,
\\
\alpha_dv^z_{\lambda,n} + \alpha_n\partial_{\bf n}v^z_{\lambda,n}=0 & \text{in} \ (0,T) \times \Gamma,
\\
v^z_{\lambda,n}(0)=z & \text{in} \ D,
\end{cases}
\eeq
and 
\beq\label{eq_app_w}
\begin{cases}
{\rm d}w^{z_1,z_2}_{\lambda,n}(t)-\nu\Delta w^{z_1,z_2}_{\lambda,n}(t)\,{\rm d}t \quad & \\
\qquad+
D\mathcal F_{\lambda,n}(u_{\lambda,n}(t))w^{z_1,z_2}_{\lambda,n}(t)\,{\rm d}t \quad& \\
\qquad+ D^2 \mathcal F_{\lambda,n}(u_{\lambda,n})[v^{z_1}_{\lambda,n}(t), v^{z_2}_{\lambda,n}(t)]\,\d t  &\\
\quad= DB_{\lambda,n}(u_{\lambda,n}(t))w^{z_1,z_2}_{\lambda,n}(t)\,{\rm d}W(t) & \\
\qquad+D^2B_{\lambda,n}(u_{\lambda,n}(t))
[v^{z_1}_{\lambda,n}(t), v^{z_2}_{\lambda,n}(t)]\,{\rm d}W(t)  & \text{in} \ (0,T) \times D,
\\
\alpha_d w^{z_1,z_2}_{\lambda,n} + \alpha_n\partial_{\bf n} w^{z_1,z_2}_{\lambda,n}=0 & \text{in} \ (0,T) \times \Gamma,
\\
w^{z_1,z_2}_{\lambda,n}(0)=0 & \text{in} \ D.
\end{cases}
\eeq

We are ready now to consider the approximated Kolmogorov equation,
which is the actual Kolmogorov equation associated to the 
doubly regularised problem \eqref{eq_app2}. 
We define the regularised Kolmogorov operator 
$(D(L_{0}^{\lambda,n}), L_{0}^{\lambda,n})$
by setting 
\[
  D(L_{0}^{\lambda,n}) := C^2_b(H)
\]
and 
\begin{align*}
  L_{0}^{\lambda,n}\varphi(x)&:=
  -\frac12\operatorname{Tr}[B_{\lambda,n}(x)^*D^2\varphi(x)B_{\lambda,n}(x)] 
  + (-\Delta x + F'_{\lambda,n}(x), D\varphi(x))_H,\\
  &\qquad x\in Z, \quad \varphi\in D(L_{0}^{\lambda,n}).
\end{align*}
Arguing exactly as in Lemma~\ref{lem:L} but for the stochastic equation \eqref{eq_app2},
it is straightforward to see that on $C^2_b(H)$ the regularised operator 
$-L_{0}^{\lambda,n}$ coincides with the infinitesimal generator of the transition semigroup 
$P^{\lambda,n}=(P^{\lambda,n}_t)_{t\ge0}$ associated to \eqref{eq_app2}, i.e.
\[
 P^{\lambda,n}_t\varphi(x):=\E\varphi(u_{\lambda,n}(t;x)), \quad x\in H, \quad \varphi\in C^2_b(H).
\]

Thanks to the regular dependence on the initial datum, 
one is able to obtain well posedness in  the classical sense at $\lambda$ and $n$ fixed.
We collect these results in the following statement.
\begin{lemma}
  \label{lem:kolm_app}
  In the current setting, there exists a positive constant $\bar\alpha$,
  independent of $\lambda$ and $n$, such that
  for every $\alpha>\bar \alpha$ there exists $C=C(\alpha)>0$
  such that the following holds: for every $g\in C^1_b(H)$,
  the function 
  \beq
  \label{sol_app}
  \varphi_{\lambda,n}(x):=\int_0^{+\infty}e^{-\alpha t}\E[g(u_{\lambda,n}(t;x))]\,\d t, 
  \quad x\in H,\quad n\in\enne,\quad\lambda>0,
  \eeq
  satisfies $\varphi_{\lambda,n}\in C^1_b(H)$ and 
  \begin{align}
    \label{est_C1}
    \norm{\varphi_{\lambda,n}}_{C^1_b(H)}\leq C\norm{g}_{C^1_b(H)}
    \quad\forall\,\lambda>0,\quad\forall\,n\in\enne.
  \end{align}
  Moreover, if the dissipativity condition \eqref{cond_uniq_inv} holds and $g\in C^2_b(H)$, then
  for every $\lambda>0$ and $n\in\enne$ it also holds that $\varphi_{\lambda,n}\in C^2_b(H)$ with 
  \beq
    \label{est_C2}
    \norm{\varphi_{\lambda,n}}_{C^2_b(H)}\leq C
    \left(1+\frac{n}{\lambda^3} + n^{\frac34\delta}\right)\norm{g}_{C^2_b(H)}
    \quad\forall\,\lambda>0,\quad\forall\,n\in\enne,
  \eeq
  and
  \beq
  \label{eq:kolm_approx}
  \alpha\varphi_{\lambda,n}(x) + L_0^{\lambda,n}\varphi_{\lambda,n}(x) = g(x) \quad\,\forall\,x\in Z.
  \eeq
\end{lemma}
\begin{proof}[Proof of Lemma~\ref{lem:kolm_app}]
It is obvious that $\varphi_{\lambda,n}\in C_b(H)$ and 
\[
\norm{\varphi_{\lambda,n}}_{C_b(H)}\le \frac1\alpha\norm{g}_{C_b(H)}.
\]
Let $x\in H$ and $u_{\lambda,n}:=u_{\lambda,n}^x$.
The It\^o formula for the square of the $H$-norm in \eqref{eq_app_v} yields, exploiting 
\eqref{est1_kolm}--\eqref{est2_kolm}, that
\begin{align}
  \nonumber
  &\frac12\norm{v^z_{\lambda,n}(t)}_H^2
  +\nu\int_0^t\norm{\nabla v^z_{\lambda,n}(s)}_H^2\,\d s\\
  \nonumber
  &\qquad\leq \frac12\norm{z}_H^2 + \left(K+\frac{C_B}2\right)
  \int_0^t\norm{v^z_{\lambda,n}(s)}_H^2\,\d s\\
  \label{ito_v}
  &\qquad+\int_0^t\left(v^z_{\lambda,n}(s), 
  DB_{\lambda,n}(u_{\lambda,n}(s))v^z_{\lambda,n}(s)\,\d W(s)\right)_H.
\end{align}
Taking expectations, we readily deduce by the Gronwall lemma that 
\beq
  \label{aux1_kolm}
  \norm{v_{\lambda,n}^z(t)}^2_{L^2(\Omega; H)} \leq e^{(2K+C_B)t}\norm{z}_H^2
  \quad\forall\,t\ge0.
\eeq
If $g\in C^1_b(H)$, recalling \eqref{reg_dep}, by the chain rule one has, for every $t\ge0$, that 
\[
  x\mapsto g(u_{\lambda,n}(t;x)) \in C^1_b(H; L^2(\Omega)),
\]
with 
\[
  D\left(x\mapsto g(u_{\lambda,n}(t;x))\right)[z] = 
  Dg(u_{\lambda,n}(t;x))v^z_{\lambda,n}(t;x), \quad x,z\in H.
\]
It follows, thanks to \eqref{aux1_kolm} that
\[
  \norm{D\left(x\mapsto g(u_{\lambda,n}(t;x))\right)}_{\mathcal L(H; L^2(\Omega))}
  \le \norm{g}_{C^1_b(H)}e^{(K+C_B/2)t} \quad\forall\,t\ge0.
\]
As soon as $\alpha>K+\frac{C_B}2$ 
(recall that $K$ and $C_B$ are independent of $\lambda$ and $n$),
the dominated convergence theorem implies that 
$\varphi_{\lambda,n}\in C^1(H)$ with 
\[
  D\varphi_{\lambda_n}(x)[z]=\int_0^{+\infty}e^{-\alpha t}
  \E\left[Dg(u_{\lambda,n}(t;x))v^z_{\lambda,n}(t;x)\right]\,\d t, \quad z\in H,
\]
and 
\[
  \norm{D\varphi_{\lambda,n}(x)}_{H} \leq 
  \norm{g}_{C^1_b(H)}\int_0^{+\infty}e^{-(\alpha-K-C_B/2)t}\,\d t
  \quad\forall\,x\in H.
\]
Choosing then $\alpha>K+\frac{C_B}2$,
this proves that actually $\varphi_{\lambda,n}\in C^1_b(H)$, 
as well as the estimate \eqref{est_C1}.\\
Let us now further assume \eqref{cond_uniq_inv} and that $g\in C^2_b(H)$.
From the It\^o formula \eqref{ito_v},
exploiting the continuous embedding $V\embed H$
and the assumption \eqref{cond_uniq_inv},
one gets that 
\begin{align*}
  &\frac12\norm{v^z_{\lambda,n}(t)}_H^2
  +\alpha_0\int_0^t\norm{v^z_{\lambda,n}(s)}_H^2\,\d s \\
  &\qquad\leq \frac12\norm{z}_H^2 + \int_0^t\left(v^z_{\lambda,n}(s), 
  DB_{\lambda,n}(u_{\lambda,n}(s))v^z_{\lambda,n}(s)\,\d W(s)\right)_H.
\end{align*}
Noting that the Burkholder-Davis-Gundy inequality
together with \eqref{est2_kolm} yield
\begin{align*}
  &\E\sup_{r\in[0,t]}\left|\int_0^r\left(v^z_{\lambda,n}(s), 
  DB_{\lambda,n}(u_{\lambda,n}(s))v^z_{\lambda,n}(s)\,\d W(s)\right)_H\right|^2\\
  &\qquad\leq 4C_B\E\int_0^t\norm{v^z_{\lambda,n}(s)}_H^4\,\d s,
\end{align*}
by raising to the square power in It\^o's formula and taking expectations we obtain 
\[
  \E\sup_{r\in[0,t]}\norm{v^z_{\lambda,n}(t)}_H^4 
  \leq 2\norm{z}_H^4 + 32C_B\int_0^t\E\norm{v^z_{\lambda,n}(s)}_H^4\,\d s.
\]
The Gronwall lemma ensures then that 
\beq
  \label{aux2_kolm}
  \norm{v_{\lambda,n}^z}_{L^4(\Omega; C([0,t]; H)}^4
  \leq 2e^{32C_B t}\norm{z}_H^4
  \quad\forall\,t\ge0.
\eeq
Similarly, using the It\^o formula for the square of the $H$-norm in \eqref{eq_app_w}, 
and exploiting now \eqref{est1_kolm}--\eqref{est5_kolm} and the Young inequality, we get 
for a constant $C$ independent of $\lambda,n$ that
\begin{align*}
  &\frac12\E\norm{w^{z_1,z_2}_{\lambda,n}(t)}_H^2
  +\nu\E\int_0^t\norm{\nabla w^{z_1,z_2}_{\lambda,n}(s)}_H^2\,\d s\\
  &\leq \left(K +\frac{C_B}2 +\frac12\right)
  \int_0^t\E\norm{w^{z_1,z_2}_{\lambda,n}(s)}_H^2\,\d s\\
  &\qquad+C\left(\frac{n^2}{\lambda^6} + n^{\frac32\delta}\right)
  \E\int_0^t\norm{v^{z_1}_{\lambda,n}(s)}_H^2
  \norm{v^{z_2}_{\lambda,n}(s)}_H^2\,\d s.
\end{align*}
Exploiting the H\"older inequality  together with \eqref{aux2_kolm}, it follows that
\begin{align*}
  &\frac12\E\norm{w^{z_1,z_2}_{\lambda,n}(t)}_H^2
  +\nu\E\int_0^t\norm{\nabla w^{z_1,z_2}_{\lambda,n}(s)}_H^2\,\d s\\
  &\leq\left(K +\frac{C_B}2 +\frac12\right)
  \int_0^t\E\norm{w^{z_1,z_2}_{\lambda,n}(s)}_H^2\,\d s\\
  &\qquad+2C\left(\frac{n^2}{\lambda^6} + 
   n^{\frac32\delta}\right) \norm{z_1}^2_H\norm{z_2}^2_H\int_0^te^{32C_Bs}\,\d s,
\end{align*}
and the Gronwall lemma ensures that, for every $t\geq0$,
\beq
  \label{aux3_kolm}
  \norm{w_{\lambda,n}(t)}^2_{L^2(\Omega; H)} \leq 
  4C\left(\frac{n^2}{\lambda^6} + n^{\frac32\delta}\right)
   t e^{(2K+33C_B+1)t}\norm{z_1}^2_H\norm{z_2}^2_H.
\eeq
Since
$g\in C^2_b(H)$, condition \eqref{reg_dep} and the chain rule give, for every $t\ge0$, that 
\[
  x\mapsto g(u_{\lambda,n}(t;x)) \in C^2_b(H; L^2(\Omega)),
\]
with 
\begin{align*}
  &D^2\left(x\mapsto g(u_{\lambda,n}(t;x))\right)[z_1,z_2] \\
  &=Dg(u_{\lambda,n}(t;x))w^{z_1,z_2}_{\lambda,n}(t;x)\\
  &\qquad+D^2g(u_{\lambda,n}(t;x))[v^{z_1}_{\lambda,n}(t;x),v^{z_2}_{\lambda,n}(t;x)],
  \quad x,z_1,z_2\in H.
\end{align*}
The estimates \eqref{aux2_kolm} and \eqref{aux3_kolm} imply then,
possibly renominating the constant $C$ independently of $\lambda$ and $n$, that
\begin{align*}
  &\norm{D^2\left(x\mapsto g(u_{\lambda,n}(t;x))\right)}_{\mathcal L_2(H\times H;  L^2(\Omega))}\\
  &\qquad\le C\norm{g}_{C^2_b(H)}\left[\left(\frac{n}{\lambda^3} + n^{\frac34\delta}\right)
  \sqrt{t}e^{\frac12(2K+33C_B+1)t} 
  +e^{16C_Bt}
  \right]\quad\forall\,t\ge0,
\end{align*}
Choosing then 
\beq\label{bar_alpha}
\bar\alpha:=\frac12\left(2K+33C_B+1\right)\vee16C_B,
\eeq
(which is independent of $\lambda$ and $n$),
the dominated convergence theorem implies that 
$\varphi_{\lambda,n}\in C^2_b(H)$ 
and, for $z_1,z_2\in H$, 
\begin{align*}
D^2\varphi_{\lambda_n}(x)[z_1,z_2]=\int_0^{+\infty}e^{-\alpha t}
  &\E\left[Dg(u_{\lambda,n}(t;x))w^{z_1,z_2}_{\lambda,n}(t;x)\right.\\
  &\quad+\left. D^2g(u_{\lambda,n}(t;x))[v^{z_1,z_2}_{\lambda,n}(t;x), v^{z_2}_{\lambda,n}(t;x)]\right]\,\d t.
\end{align*}
It follows for every $x\in H$ that
\[
  \norm{D^2\varphi_{\lambda,n}(x)}_{\mathcal L(H,H)} \leq 
  C\left(\frac{n}{\lambda^3} + n^{\frac34\delta}\right)
  \norm{g}_{C^2_b(H)}\int_0^{+\infty}(\sqrt{t} + 1)e^{-(\alpha-\bar\alpha)t}\,\d t,
\]
hence for $\alpha>\bar\alpha$ this shows that
$\varphi_{\lambda,n}\in C^2_b(H)$, as well as the estimate \eqref{est_C2}.\\
Eventually, let us show that \eqref{eq:kolm_approx} holds.
To this end, we readily note from the stochastic equation \eqref{eq_app2} 
that for every $x\in Z$ we have in particular that $u_{\lambda,n}\in L^2(\Omega; L^2(0,T; H))$.
Hence,  the fact that 
$g\in C^2_b(H)$, It\^o's formula, and the definition of $L_0^{\lambda,n}$ readily give for every $t\geq0$ that 
\begin{align*}
  &\E g(u_{\lambda,n}(t)) + \int_0^t\E (L_0^{\lambda,n}g)(u_{\lambda,n}(s))\,\d s
  =g(x).
\end{align*}
Since $-L_{0}^{\lambda,n}$ coincides on $C^2_b(H)$ 
with the infinitesimal generator of the transition semigroup $P^{\lambda,n}$
associated to \eqref{eq_app2}
(so in particular $L_0^{\lambda,n}$ and $P^{\lambda,n}$ commute on $C^2_b(H)$), it holds in particular that 
\begin{align*}
  \int_0^t\E (L_0^{\lambda,n}g)(u_{\lambda,n}(s))\,\d s
  &=\int_0^tP_s^{\lambda,n}(L_0^{\lambda,n}g)(x)\,\d s=
  \int_0^tL_0^{\lambda,n}(P_s^{\lambda,n}g)(x)\,\d s\\
  &=L_0^{\lambda,n}\int_0^tP_s^{\lambda,n}g(x)\,\d s
  =L_0^{\lambda,n}\int_0^t\E g(u_{\lambda,n}(s))\,\d s,
\end{align*}
from which we get that
\begin{align*}
  &e^{-\alpha t}\E g(u_{\lambda,n}(t)) + \alpha\int_0^te^{-\alpha s}\E g(u_{\lambda,n}(s))\,\d s\\
  &\qquad+ L_0^{\lambda,n}\int_0^te^{-\alpha s}\E g(u_{\lambda,n}(s))\,\d s
  =g(x).
\end{align*}
By boundedness of $g$, letting $t\to+\infty$ yields \eqref{eq:kolm_approx}.
This concludes the proof.
\end{proof}

\subsection{Well posedness \`a la Friedrichs}
We are now ready to show that the Kolmogorv equation \eqref{eq:kolm}
is well posed in the sense of Friedrichs, as rigorously specified in Proposition~\ref{prop:class} below.
This will allow to fully characterise the infinitesimal generator 
of the transition semigroup $P$ on $L^2(\A,\mu)$ in terms of the Kolmogorov operator.
\begin{proposition}
  \label{prop:class}
  In the current setting, assume the dissipativity condition \eqref{cond_uniq_inv},
  let $\bar\alpha$ be as in \eqref{bar_alpha}, and let $\alpha>\bar\alpha$.
  Then, for every $g\in L^2(\A,\mu)$ there exist a unique $\varphi\in L^2(\A,\mu)$ and
  two sequences $\{g_m\}_{m\in\enne}\subset L^2(\A,\mu)$ and $\{\varphi_m\}_{m\in\enne}\subset D(L_0)$
  such that 
  \[
  \alpha\varphi_m + L_0\varphi_m = g_m \quad\mu\text{-a.s.~in } \A,\quad\forall\,m\in\enne,
  \]
  and, as $m\to\infty$,
  \[
  \varphi_m\to \varphi\quad \text{in } L^2(\A,\mu), \qquad
  g_m\to g\quad \text{in } L^2(\A,\mu).
  \]
  In particular, the range of $\alpha I + L_0$ is dense in $L^2(\A,\mu)$.
\end{proposition}
\begin{proof}[Proof of Proposition~\ref{prop:class}]
Given $g\in L^2(\A,\mu)$, we define $\tilde g:H\to\erre$ by extending $g$ to zero outside $\A$, namely
\[
  \tilde g(x):=
  \begin{cases}
  g(x) \quad&\text{if } x\in\A,\\
  0 &\text{if } x\in H\setminus\A.
  \end{cases}
\]
Analogously, note that the probability measure $\mu\in\cP(\A)$ extends (uniquely) to 
a probability measure $\tilde \mu\in\cP(H)$, by setting
\[
  \tilde \mu(E):= \mu(E\cap\A), \quad E\in\cB(H).
\]
With this notation, it is clear that $\tilde g\in L^2(H,\tilde \mu)$: 
by density of $C^2_b(H)$ in $L^2(H,\tilde \mu)$, there exists a sequence $\{\tilde f_j\}_j\subset C^2_b(H)$
such that (see Lemma~\ref{lem:density})
\beq
\label{approx_fj}
  \lim_{j\to\infty}\|\tilde f_j-\tilde g\|_{L^2(H,\tilde\mu)}=0.
\eeq
Clearly, setting $f_j:=(\tilde f_j)_{|\A}$ for every $j\in\enne$, one has that 
$\{f_j\}_{j\in\enne}\subset D(L_0)$ by definition of $D(L_0)$, and also, thanks to the definition of $\tilde \mu$,
\beq
\label{fri1}
  \lim_{j\to\infty}\norm{f_j- g}_{L^2(\A,\mu)}=0.
\eeq
Let $j\in\enne$ be fixed: for every $\lambda>0$ and $n\in\enne$ we set
\[
  \tilde \varphi_{\lambda,n,j}(x):=
  \int_0^{+\infty}e^{-\alpha t}\E[\tilde f_j(u_{\lambda,n}(t;x))]\,\d t, 
  \quad x\in H,
\]
and 
\[
  \varphi_{\lambda,n,j}:=(\tilde\varphi_{\lambda,n,j})_{|\A}.
\]
Lemma~\ref{lem:kolm_app} ensures, 
for all $\lambda>0$ and $n\in\enne$, that $\tilde \varphi_{\lambda,n,j} \in C^2_b(H)$,
hence in particular that $\varphi_{\lambda,n,j} \in D(L_0)$, and 
\[
  \alpha\tilde\varphi_{\lambda,n,j}(x) 
  + L_0^{\lambda,n}\tilde\varphi_{\lambda,n,j}(x) = \tilde f_j(x) \quad\,\forall\,x\in H.
\]
It follows that, for every $x\in\A_{str}$,
\beq
  \label{fri_00}
  \alpha\varphi_{\lambda,n,j}(x) + L_0\varphi_{\lambda,n,j}(x) = 
  f_j(x) + L_0\varphi_{\lambda,n,j}(x) - L_0^{\lambda,n}\tilde\varphi_{\lambda,n,j}(x).
\eeq
Let now $j\in\enne$ be fixed. For every $x\in\A_{str}$ one has that 
\begin{align*}
  &L_0\varphi_{\lambda,n,j}(x) - L_0^{\lambda,n}\tilde\varphi_{\lambda,n,j}(x)\\
  &=-\frac12\operatorname{Tr}\left[B(x)^*D^2\varphi_{\lambda,n,j}(x)B(x) - 
  B_{\lambda,n}(x)^*D^2\varphi_{\lambda,n,j}(x)B_{\lambda,n}(x)\right]\\
  &\qquad+\left(F'(x)-\mathcal F_{\lambda,n}(x), D\varphi_{\lambda,n,j}(x)\right)_H\\
  &=-\frac12\operatorname{Tr}\left[(B(x)^*-B_{\lambda,n}(x)^*)D^2\varphi_{\lambda,n,j}(x)B(x)\right]\\
  &\qquad-\frac12\operatorname{Tr}\left[B_{\lambda,n}(x)^*D^2\varphi_{\lambda,n,j}(x)(B(x)-B_{\lambda,n}(x))\right]\\
  &\qquad+\left(F'(x)-\mathcal F_{\lambda,n}(x), D\varphi_{\lambda,n,j}(x)\right)_H.
\end{align*}
It follows from Lemma~\ref{lem:kolm_app} that there exists a constant $C>0$, 
which is independent of $\lambda$, $n$, and $j$, such that 
\begin{align}
 \nonumber
  &|L_0\varphi_{\lambda,n,j}(x) - L_0^{\lambda,n}\tilde\varphi_{\lambda,n,j}(x)|\\
  \nonumber
  &\leq \norm{\varphi_{\lambda,n,j}}_{C^2_b(H)}
  \norm{B_{\lambda,n}(x)-B(x)}_{\mathcal L_{HS}(U,H)}^2 \\
  \nonumber
  &\qquad+\norm{\varphi_{\lambda,n,j}}_{C^1_b(H)}\norm{\mathcal F_{\lambda,n}(x) - F'(x)}_H\\
  \nonumber
  &\leq C\|\tilde f_j\|_{C^2_b(H)}
  \left(1+\frac{n}{\lambda^3} + n^{\frac34\delta}\right)\norm{B_{\lambda,n}(x)-B(x)}_{\mathcal L_{HS}(U,H)}^2\\
  \label{diff_kolm}
  &\qquad
  + C\|\tilde f_j\|_{C^1_b(H)}\norm{\mathcal F_{\lambda,n}(x) - F'(x)}_H.
\end{align}
Now, let us estimate the two terms on the right hand side. As for the first one, 
we have that 
\begin{alignat*}{2}
  \norm{B_{\lambda,n}(x)-B(x)}_{\mathcal L_{HS}(U,H)}^2
  &\leq3\norm{B_{\lambda,n}(x)-e^{-\frac{\mathcal C}{n^\delta}}B_\lambda(x)}_{\mathcal L_{HS}(U,H)}^2
  \quad&&=:3I_1\\
  &+3\norm{e^{-\frac{\mathcal C}{n^\delta}}B_{\lambda}(x)-
  e^{-\frac{\mathcal C}{n^\delta}}B(x)}_{\mathcal L_{HS}(U,H)}^2
  \quad&&=:3I_2\\
  &+3\norm{e^{-\frac{\mathcal C}{n^\delta}}B(x)-B(x)}_{\mathcal L_{HS}(U,H)}^2
  \quad&&=:3I_3.
\end{alignat*}
Exploiting the definition of $B_{\lambda,n}$, the Jensen inequality, and
the fact that $B_\lambda$ is $\sqrt{C_B}$-Lipschitz continuous on $H$, we have
\begin{align*}
  I_1&=\norm{B_{\lambda,n}(x)-e^{-\frac{\mathcal C}{n^\delta}}B_\lambda(x)}_{\mathcal L_{HS}(U,H)}^2\\
  &=\norm{
  \int_He^{-\frac{\mathcal C}{n^\delta}}\left(B_\lambda(e^{-\frac{\mathcal C}{n^\delta}}x+y) - 
  B_\lambda(x)\right)\,N_{Q_{1/n^\delta}}(\d y)}_{\mathcal L_{HS}(U,H)}^2\\
  &\leq\int_H
  \norm{B_\lambda(e^{-\frac{\mathcal C}{n^\delta}}x+y) - 
  B_\lambda(x)}_{\mathcal L_{HS}(U,H)}^2\,N_{Q_{1/n^\delta}}(\d y)\\
  &\leq C_B\int_H\norm{e^{-\frac{\mathcal C}{n^\delta}}x +y - x}_H^2\,N_{Q_{1/n^\delta}}(\d y)\\
  &\leq 2C_B\norm{e^{-\frac{\mathcal C}{n^\delta}}x - x}^2_H + 2C_B\int_H\norm{y}_H^2\,N_{Q_{1/n^\delta}}(\d y)\\
  &\leq 4C_B\norm{x}_H\norm{e^{-\frac{\mathcal C}{n^\delta}}x - x}_H 
  + 2C_B\int_H\norm{y}_H^2\,N_{Q_{1/n^\delta}}(\d y)
\end{align*}
from which we get, using definition  \eqref{Qt}, estimate \eqref{est_Qt}, and 
the fact that $x\in\A_{str}\subset D(\mathcal C)=Z$ with $\norm{x}_H\leq|D|^{1/2}$,
\begin{align*}
  I_1&\leq 4C_B\norm{x}_H\int_0^{1/n^\delta}\norm{e^{-s\mathcal C}\mathcal C x}_H\,\d s+
  2C_B\operatorname{Tr}(Q_{1/n^\delta})\\
  &\leq \frac{4C_B|D|^{1/2}}{n^{\delta}}\norm{x}_Z
  +\frac{2C_B}{n^\delta}\norm{\mathcal C^{-1}}^2_{\mathcal L_{HS}(H,H)}
\end{align*}
Possibly renominating $C$ independently of $\lambda$, $n$, and $j$, this shows that 
\beq
  \label{fri2}
  I_1\leq \frac{C}{n^\delta}\left(1+\norm{x}_Z\right).
\eeq
As far as $I_2$ is concerned, it is immediate to see, thanks to 
the non-expansivity of $e^{-t\mathcal C}$ and the definition \eqref{def_app2}, that 
\begin{align}
  \nonumber
  I_2&\leq\norm{B_{\lambda}(x)-B(x)}_{\mathcal L_{HS}(U,H)}^2
  =\sum_{k\in\enne}|(\tilde h_{k}\star \rho_{\lambda^\gamma})(x) - h_k(x)|^2\\
  \label{fri3}
  &= \sum_{k\in\enne}|(\tilde h_{k}\star \rho_{\lambda^\gamma})(x) - \tilde h_k(x)|^2
  \leq {\lambda^{2\gamma}}\sum_{k\in\enne}\norm{\tilde h_k'}_{L^\infty(\erre)}^2
  \leq {C_B}{\lambda^{2\gamma}}.
\end{align}
Also, we have by the contraction of $e^{-\frac{\mathcal C}{n^\delta}}$ and the H\"older inequality that 
\begin{align*}
  I_3&=\sum_{k\in\enne}\norm{e^{-\frac{\mathcal C}{n^\delta}}h_k(x) - h_k(x)}_H^2
  \leq2\sum_{k\in \enne}\norm{h_k(x)}_H\norm{e^{-\frac{\mathcal C}{n^\delta}}h_k(x) - h_k(x)}_H\\
  &\leq2C_B^{1/2}\left(\sum_{k\in\enne}\norm{\int_0^{1/n^\delta}e^{-s\mathcal C}\mathcal{C}h_k(x)\,\d s}_H^2\right)^{1/2}\\
  &\leq\frac{2C_B^{1/2}}{n^{\delta/2}}
  \left(\sum_{k\in\enne}\int_0^{1/n^\delta}\norm{e^{-s\mathcal C}\mathcal Ch_k(x)}^2_H\,\d s\right)^{1/2}
  \leq\frac{2C_B^{1/2}}{n^{\delta}}\left(\sum_{k\in\enne}\norm{\mathcal Ch_k(x)}_H^2\right)^{1/2},
\end{align*}
where, by definition of $\mathcal C$ and the regularity of $\{h_k\}_k$, it holds that 
\begin{align*}
  \norm{\mathcal Ch_k(x)}_H &\leq \norm{h_k(x)}_H + \norm{h_k'(x)\Delta x}_H + 
  \norm{h_k''(x)|\nabla x|^2}_H\\
  &\leq \norm{h_k}_{C^2([-1,1])}\left( |D|^{1/2} + \norm{\Delta x}_H + \norm{|\nabla x|^2}_H\right).
\end{align*}
Noting that $H^2(D)\embed W^{1,4}(D)\embed L^\infty(D)$,
the Gagliardo-Nierenberg interpolation inequality yields that 
\[
  \norm{\nabla y}_{L^4(D)} \lesssim_{D,d} \norm{y}_{H^2(D)}^{1/2}\norm{y}_{L^\infty(D)}^{1/2} + \norm{y}_{L^1(D)}
  \quad\forall\,y\in H^2(D),
\]
hence, recalling that $\norm{x}_{L^\infty(D)}\leq1$ for all $x\in\A_{str}$, one has
\[
  \norm{|\nabla x|^2}_H = \norm{\nabla x}_{L^4(D)}^2
  \lesssim_{D} \norm{x}_{H^2(D)}\norm{x}_{L^\infty(D)} + \norm{x}_{L^1(D)} \lesssim \norm{x}_{H^2(D)} + 1.
\]
Putting this information together, by {\bf H2'} we deduce that 
there exists a positive constant $C$ independent of $\lambda$, $n$, and $j$, such that 
\beq
  \label{fri4}
  I_3 \leq \frac{C}{n^{\delta}}\left(1+\norm{x}_Z\right)\left(\sum_{k\in\enne}\norm{h_k}^2_{C^2([-1,1])}\right)^{1/2}
  \leq \frac{C}{n^{\delta}}\left(1+\norm{x}_Z\right).
\eeq
Going back to \eqref{diff_kolm}, we infer that, possibly renominating the constant $C$
independently of $\lambda$, $n$, and $j$, 
\begin{align}
  \nonumber
  &|L_0\varphi_{\lambda,n,j}(x) - L_0^{\lambda,n}\tilde\varphi_{\lambda,n,j}(x)|\\
  \nonumber
  &\leq C\|\tilde f_j\|_{C^2_b(H)}
  \left(1+\frac{n}{\lambda^3} + n^{\frac34\delta}\right)\left(\frac1{n^\delta}
 + \lambda^{2\gamma}\right)\left(1+\norm{x}_Z\right)\\
  \label{diff2_kolm}
  &\qquad+ C\|\tilde f_j\|_{C^1_b(H)}\norm{\mathcal F_{\lambda,n}(x) - F'(x)}_H.
\end{align}
As for the second term on the right hand side, we proceed as above, getting
\begin{alignat*}{2}
  \norm{\mathcal F_{\lambda,n}(x) - F'(x)}_H
  &\leq\norm{\mathcal F_{\lambda,n}(x)-e^{-\frac{\mathcal C}n}\mathcal F_\lambda(x)}_{H}
  \quad&&=:J_1\\
  &+\norm{e^{-\frac{\mathcal C}n}\mathcal F_{\lambda}(x)-
  e^{-\frac{\mathcal C}n}F'(x)}_{H}
  \quad&&=:J_2\\
  &+\norm{e^{-\frac{\mathcal C}{n}}F'(x)-F'(x)}_{H}
  \quad&&=:J_3.
\end{alignat*}
The analogous computations as the term $I_1$ above imply, 
using the definition \eqref{def_app3}, the $\frac1\lambda$-Lipschitz continuity of $\mathcal F_\lambda:H\to H$,
and \eqref{Qt}--\eqref{est_Qt}, that
\begin{align}
  \nonumber
  J_1&\leq \frac1\lambda\int_H\norm{e^{-\frac{\mathcal C}n}x + y - x}_H\, N_{Q_{1/n}}(\d y)\\
  \nonumber
  &\leq\frac1\lambda\int_0^{1/n}\norm{e^{-\frac{\mathcal C}n}\mathcal Cx}_H\,\d s + 
  \frac1\lambda\int_H\norm{y}_H\, N_{Q_{1/n}}(\d y)\\
  \label{fri5}
  &\leq\frac1{\lambda n}\norm{x}_Z + \frac1\lambda\sqrt{\operatorname{Tr}(Q_{1/n})}
  \leq\frac1{\lambda n}\norm{x}_Z + \frac1{\lambda\sqrt{n}}.
\end{align}
Furthermore, the definition of Yosida approximation
and \eqref{def_app1}, together with the contraction of $e^{-\frac{\mathcal C}n}$, yield 
\begin{align}
  \nonumber
  J_2&\leq\norm{\mathcal F_{\lambda}(x) - F'(x)}_H=\norm{(\rho_{\lambda^2}\star\beta_\lambda)(x)-\beta(x)}_H\\
  \nonumber
  &\leq\norm{(\rho_{\lambda^2}\star\beta_\lambda)(x)-\beta_\lambda(x)}_H
  +\norm{\beta_\lambda(x)-\beta(x)}_H\\
  &\leq \lambda^2\frac1\lambda + \lambda\norm{J_\lambda(x)}_H  \leq\lambda\left(1+\norm{x}_H\right).
  \label{fri6}
\end{align}
Hence, exploiting \eqref{approx_fj}, \eqref{fri5}, and \eqref{fri6} in \eqref{diff2_kolm}, 
possibly renominating the constant $C$  independently of $\lambda$, $n$, and $j$, we obtain that 
\begin{align}
  \nonumber
  &|L_0\varphi_{\lambda,n,j}(x) - L_0^{\lambda,n}\tilde\varphi_{\lambda,n,j}(x)|\\
  \nonumber
  &\leq C\|\tilde f_j\|_{C^2_b(H)} \left(1+\norm{x}_Z\right)
  \left(1+\frac{n}{\lambda^3} + n^{\frac34\delta}\right)
  \left(\frac1{n^\delta}+ \lambda^{2\gamma}\right) \\
  \label{diff3_kolm}
  &\quad+ C\|\tilde f_j\|_{C^1_b(H)}
  \left[\left(1+\norm{x}_Z\right)
  \left(\lambda + \frac1{\lambda\sqrt{n}}\right) +
  \norm{e^{-\frac{\mathcal C}{n}}F'(x)-F'(x)}_{H} \right] ,
\end{align}
where the constant $C$ is independent of $\lambda$, $n$, and $j$.
Choosing the specific sequence $\lambda_n:=n^{-1/4}$ in \eqref{diff3_kolm}, 
we deduce for every $j,n\in\enne$ that 
\begin{align*}
  &|L_0\varphi_{\lambda_n,n,j}(x) - L_0^{\lambda_n,n}\tilde\varphi_{\lambda_n,n,j}(x)|\\
  &\leq C\|\tilde f_j\|_{C^2_b(H)} \left(1+\norm{x}_Z\right) 
  \left(1+n^{\frac74} + n^{\frac34\delta}\right)\left(\frac1{n^\delta} + \frac1{n^{\frac\gamma2}}\right)\\
  &\quad+ C\|\tilde f_j\|_{C^1_b(H)}\left[
  \frac2{n^{\frac14}}\left(1+\norm{x}_Z\right)
  +\norm{e^{-\frac{\mathcal C}{n}}F'(x)-F'(x)}_{H}\right].
\end{align*}
At this point, if we choose the rate coefficients $\gamma$ and $\delta$ is such a way that 
\[
  \delta>\frac74, \quad \gamma>\frac72, \quad \gamma>\frac32\delta,
\]
by setting for example
\[
  \gamma:=4, \qquad \delta:=2,
\]
renominating $C$ independently of $n$ and $j$
it easily follows  that
\begin{align*}
  &|L_0\varphi_{\lambda_n,n,j}(x) - L_0^{\lambda_n,n}\tilde\varphi_{\lambda_n,n,j}(x)|\\
  &\leq \frac{C}{n^{\frac14}}\|\tilde f_j\|_{C^2_b(H)} \left(1+\norm{x}_Z\right) 
 + C\|\tilde f_j\|_{C^1_b(H)}\norm{e^{-\frac{\mathcal C}{n}}F'(x)-F'(x)}_{H}.
\end{align*}
Since the right-hand side belongs to $L^2(\A,\mu)$ thanks to Proposition~\ref{supp1},
this yields integrating with respect to $\mu$ and renominating the constant $C$
as usual that
\begin{align}
  \nonumber
  &\norm{L_0\varphi_{\lambda_n,n,j} - L_0^{\lambda_n,n}\tilde\varphi_{\lambda_n,n,j}}_{L^2(\A,\mu)}^2\\
  &\leq \frac{C}{n^{\frac12}}\|\tilde f_j\|_{C^2_b(H)}^2
 + C\|\tilde f_j\|_{C^1_b(H)}^2\int_H\norm{e^{-\frac{\mathcal C}{n}}F'(x)-F'(x)}_{H}^2\,\mu(\d x)
 \label{diff4_kolm}
\end{align}
We are ready now to construct the sequence $\{g_m\}_{m\in\enne}$.
Let $m\in\enne$ be arbitrary. By virtue of \eqref{fri1}, we can 
pick $j_m\in\enne$ such that 
\[
  \|f_{j_m}-g\|_{L^2(\A,\mu)}\le\frac1m.
\]
Also, noting that Proposition~\ref{supp1} and the dominated convergence theorem imply 
\[
  \lim_{n\to\infty}\int_\A\norm{e^{-\frac{\mathcal C}{n}}F'(x)-F'(x)}_{H}^2\,\mu(\d x) = 0,
\]
given such $j_m$ we can then choose $n_m\in\enne$ sufficiently large such that 
\[
  \frac{1}{n_m^{\frac12}}\|\tilde f_{j_m}\|_{C^2_b(H)}^2\leq\frac1{m^2}
\]
and
\[
  \|\tilde f_{j_m}\|_{C^1_b(H)}^2\int_H\norm{e^{-\frac{\mathcal C}{n_m}}F'(x)-F'(x)}_{H}^2\,\mu(\d x)\leq\frac1{m^2}.
\]
Setting then 
\begin{align*}
  \varphi_m&:=\varphi_{\lambda_{n_m}, n_m, j_m}\in D(L_0), \\
  g_m&:= f_{j_m} + L_0\varphi_{\lambda_{n_m},n_m,j_m} 
  - L_0^{\lambda_{n_m},n_m}\tilde\varphi_{\lambda_{n_m},n_m,j_m} \in L^2(\A,\mu),
\end{align*}
thanks to \eqref{fri_00} one has exactly 
\[
  \alpha\varphi_m + L_0\varphi_m = g_m \quad\mu\text{-a.s.~in } \A,\quad\forall\,m\in\enne,
\]
while the estimate \eqref{diff4_kolm} yields, by the choices made above,
\[
  \norm{g_m-g}_{L^2(\A,\mu)} \leq \frac{C}m \longrightarrow0 \quad\text{as } m\to\infty.
\]
Also, we note that $L$ is accretive in $L^2(\A,\mu)$ because it is the infinitesimal
generator of the semigroup of contractions $P$ on $L^2(\A,\mu)$:
hence, since by Lemma~\ref{lem:L} we know that $L_0=L$ on $D(L_0)$, 
it is immediate to deduce that 
\[
  \alpha\norm{\varphi_{m_1}-\varphi_{m_2}}_{L^2(\A,\mu)}
  \leq \norm{g_{m_1}-g_{m_2}}_{L^2(\A,\mu)} \quad\forall\,m_1,m_2\in\enne.
\]
It follows that $\{\varphi_m\}_{m\in\enne}$ is Cauchy in $L^2(\A,\mu)$,
hence it converges to some $\varphi\in L^2(\A,\mu)$. It is not difficult to 
see by using again the accretivity that $\varphi$ is unique, in the sense that 
it does not depend on the sequences $\{\varphi_m\}_{m\in\enne}$ and $\{g_m\}_{m\in\enne}$.
This finally concludes the proof of Proposition~\ref{prop:class}.
\end{proof}

We are now ready to state the main result of this section,
which completely characterises the infinitesimal generator of the transition 
semigroup $P$ on $L^2(\A,\mu)$ in terms of the Kolmogorov operator $L_0$.
\begin{theorem}
  \label{th:ident}
  In the current setting, assume the dissipativity condition \eqref{cond_uniq_inv}.
  Then, the Kolmogorov operator $L_0$ is closable in $L^2(\A,\mu)$,
  and its closure $\overline{L_0}$ coincides with the infinitesimal generator $L$
  of the transition semigroup $P$ on $L^2(\A,\mu)$.
\end{theorem}
\begin{proof}
  Since $P$ is a semigroup of contractions in $L^2(\A,\mu)$, its infinitesimal generator 
  $L$ is m-accretive in $L^2(\A,\mu)$. Since by Lemma~\ref{lem:L}, we know that 
  $L_0=L$ in $D(L_0)$, it follows that $L_0$ is accretive in $L^2(\A,\mu)$, hence closable.
  Let $(\overline{L_0}, D(\overline{L_0})$ denote such closure
  and let $\alpha>\bar\alpha$, where $\bar\alpha$ is given as in \eqref{bar_alpha}.
  By the Lumer-Philips theorem, the range of $\alpha I + \overline{L_0}$
  coincides with the closure in $L^2(\A,\mu)$ of the range of $\alpha I + L_0$.
  Since the range of $\alpha I + L_0$ is dense in $L^2(\A,\mu)$ by Proposition~\ref{prop:class},
  it follows that $(\overline{L_0}, D(\overline{L_0})$ is m-accretive in $L^2(\A,\mu)$,
  hence it generates a strongly continuous semigroup of contractions in $L^2(\A,\mu)$.
  Since $D(L_0)$ is a core for $\overline{L_0}$ and $L=\overline{L_0}$ on $D(L_0)$,
  it follows that $(L,D(L))=(\overline{L_0}, D(\overline{L_0}))$.
\end{proof}

\appendix
\section{A priori estimates}
\label{app1}

We collect here the needed apriori estimates on the solution process.
\begin{lemma}
\label{tight}
Assume {\bf H1--H2}. Then, for every initial datum $x \in \A$
the respective variational solution to equation \eqref{eq_AC} satisfies 
\begin{equation*}
\E\sup_{r\in[0,t]}\|u(r;x)\|_H^2 + 
\int_0^t\mathbb{E}\left[\|u(s;x)\|^2_V\right]\, {\rm d}s  
\lesssim_{C_1, C_B, |D|,\nu}\left(\|x\|^2_H+t\right).
\end{equation*}
\end{lemma}
\begin{proof}
Let $u:=u^x$ be the unique variational solution to equation \eqref{eq_AC} starting from $x$.
The It\^o formula for the squared $H$-norm $\|\cdot\|^2_H$ yields, for every $t\ge 0$, $\mathbb{P}$-almost surely,
\begin{align}
\label{stima1}
\frac 12 \|u(t)\|^2_H&+\nu\int_0^t \|\nabla u(s)\|^2_H\, {\rm d}s + \int_0^t (F'(u(s)),u(s))_H\, {\rm d}s 
\notag\\
&=\frac 12 \|x\|^2_H + \int_0^t (u(s), B(u(s)){\rm d}W(s))_H 
+ \frac 12 \int_0^t \|B(u(s))\|^2_{\mathcal L_{HS}(U,H)}\, {\rm d}s.
\end{align}
By means of \eqref{F'_prop} we estimate
\begin{align*}
\int_0^t (F'(u),u)_H\, {\rm d}s \ge C_0 \int_0^t \|u(s)\|^2_H\, {\rm d}s -C_1t,
\end{align*}
while from \eqref{HS_norm} we immediately get
\begin{align*}
\frac 12 \int_0^t \|B(u(s))\|^2_{\mathcal L_{HS}(U,H)}\, {\rm d}s \le \frac{C_B|D|t}{2}.\end{align*}
Hence, taking the supremum in time and expectations in \eqref{stima1} we have 
\begin{align*}
&\frac 12 \E\sup_{r \in [0,t]}\|u(r)\|^2_H+\nu\int_0^t \E\|\nabla u(s)\|^2_H\, {\rm d}s
\notag\\
&\le\frac 12 \|x\|^2_H + \left( \frac{C_B|D|}{2}+C_1\right) t
+\E\sup_{r \in [0,t]}\int_0^r (u(s), B(u(s)){\rm d}W(s))_H.
\end{align*}
By means of the Burkholder-Davis-Gundy and Young inequalities
(see \cite[Lem.~4.3]{MS-var} and \cite[Lem.~4.1]{MS-ref} for details), 
we estimate the last term in the above expression as 
\begin{align*}
\E\sup_{r \in [0,t]}\int_0^r (u(s), B(u(s)){\rm d}W(s))_H
&\le \frac 14 \mathbb{E} \sup_{r \in [0,t]} \|u(r)\|^2_H +Ct,
\end{align*}
where the constant $C$ depends only on $C_B$ and $|D|$ (not on $t$).
Combining the above estimates, the thesis follows.
\end{proof}

\begin{lemma}
\label{H^2_est}
Assume {\bf H1--H2}. Then, for every initial datum $x \in \A\cap V$
the respective analytically strong solution to equation \eqref{eq_AC} satisfies 
\begin{equation*}
\E\sup_{r\in[0,t]}\|u(r;x)\|_V^2
+\int_0^t\mathbb{E}\|u(s;x)\|^2_{Z}\, {\rm d}s \lesssim_{K,C_0,C_1, C_B,|D|,\nu}\left(\|x\|^2_V+ t\right).
\end{equation*}
\end{lemma}
\begin{proof}
Let $u:=u^x$ be the unique analytically strong solution to equation \eqref{eq_AC} starting from $x$.
We apply the It\^o formula in \cite[Theorem 4.2]{Pardoux_PhD}
(see also \cite[Prop.~3.3]{SS}) to the functional 
$\frac12\|\nabla\cdot\|_H^2$. We obtain, for any $t\ge 0$, $\mathbb{P}$-almost surely,
\begin{align}
\label{est_2}
&\frac12 \|\nabla u(t)\|^2_H
+ \nu\int_0^t \|\Delta u(s)\|^2_H\, {\rm d}s + \int_0^t\int_D F''(u(s))|\nabla u(s)|^2 \, {\rm d}s
\notag \\
&=\frac 12 \|\nabla x\|^2_H + 
\frac 12 \int_0^t \sum_{k\in\enne}\int_D|h_k'(u(s))\nabla u(s)|^2\,\d s
\notag \\
&\qquad- \int_0^t \left( \Delta u(s), B(u(s))\,{\rm d}W(s)\right)_H.
\end{align}
By {\bf H1} we have that
\begin{align*}
\int_0^t\int_D F''(u(s))|\nabla u(s)|^2 \, {\rm d}s \ge -K\int_0^t\|\nabla u(s)\|_H^2\,\d s,
\end{align*}
while assumption {\bf H2} yields
\begin{align*}
\frac 12 \int_0^t \sum_{k \in \mathbb{N}}\int_D |h_k'(u(s))\nabla u(s)|^2\, {\rm d}s \le \frac{C_B}{2}\int_0^t \|\nabla (u(s))\|^2_H\, {\rm d}s.
\end{align*}
Using again the Burkholder-Davis-Gundy and Young inequalities
(as in \cite[Lem.~4.3]{MS-var}) together with \eqref{HS_norm} we get that 
\begin{align*}
\E\sup_{r \in [0,t]}\int_0^r (\Delta u(s), B(u(s)){\rm d}W(s))_H
&\le \frac \nu2 \mathbb{E} \int_0^t \|\Delta u(s)\|^2_H\,\d s +Ct,
\end{align*}
where the constant $C$ depends only on $\nu$, $C_B$, and $|D|$ (not on $t$).
Hence, taking supremum in time and expectations in \eqref{est_2}, rearranging the terms we obtain 
\begin{align*}
&\frac12 \E\sup_{r\in[0,t]}\left[\|\nabla u(r)\|^2_H\right]+
 \frac\nu2\int_0^t \E \|\Delta u(s)\|^2_H\, {\rm d}s\\
&\le \frac 12 \|\nabla x\|^2_H+ \left(K+\frac{C_B}{2} \right)\int_0^t \mathbb{E} \left[ \|\nabla u(s)\|^2_H\right]\, {\rm d}s
+Ct,
\end{align*}
and the thesis follows from Lemma \ref{tight}.
\end{proof}

\begin{lemma}
\label{lem:F'}
Assume {\bf H1--H2}. Then, for every initial datum $x \in \A\cap V$
the respective analytically strong solution to equation \eqref{eq_AC} satisfies 
\begin{equation*}
\int_0^t\mathbb{E}\|F'(u(s;x))\|^2_{H}\, {\rm d}s \lesssim_{K,C_0,C_1, C_B,|D|,\nu}\left(1+ t\right).
\end{equation*}
\end{lemma}
\begin{proof}
As a consequence of It\^o formula on suitable 
Yosida-type approximations of $F'$ (see \cite[Sec.~4.2]{Ber} for details)
it holds for every $t\ge0$ that 
\begin{align*}
  &\E \int_DF(u(t;x)) + \nu\E\int_0^t\int_DF''(u(s;x))|\nabla u(s;x)|^2\,\d s + \E\int_0^t\norm{F'(u(s;x))}_H^2\,\d s\\
  &\le \int_DF(x) + \frac12\E\int_0^t\sum_{k\in\enne}\int_DF''(u(s;x))|h_k(u(s;x))|^2\,\d s.
\end{align*}
Combining then assumptions {\bf H1--H2} yields
\begin{align*}
  &E\int_0^t\norm{F'(u(s;x))}_H^2\,\d s
  \le|D| \norm{F}_{C([-1,1])} + K\nu\int_0^t\norm{\nabla u(s;x)}_H^2\,\d s + \frac{C_B|D|}2 t,
\end{align*}
and the thesis follows from Lemma~\ref{H^2_est}.
\end{proof}

\section{A density result}
\label{app2}
\begin{lemma}
  \label{lem:density}
  Let $\tilde\mu\in\cP(H)$ and $\tilde g\in L^2(H,\tilde\mu)$. Then, there exists 
  a sequence $\{\tilde f_j\}_{j\in\enne}\subset C^2_b(H)$ such that 
  \[
  \lim_{j\to\infty}\norm{\tilde f_j-\tilde g}_{L^2(H,\tilde\mu)}=0.
  \]
\end{lemma}
\begin{proof}
  For every $i\in\enne$ we define $T_i:\erre\to\erre$ as $T_i(r):=\max\{-1,\min\{r,1\}\}$, $r\in\erre$.
  Then, for every $\ell\in\enne$ we set 
  \[
  \tilde f_{i,\ell}(x):=R_{1/\ell}T_{i}(\tilde g)(x)=
  \int_H T_{i}(\tilde g(e^{-\frac{\mathcal C}\ell}x + y))\,N_{Q_{1/\ell}}(\d y), \quad x\in H.
  \]
  Since $T_{i}(\tilde g)\in\cB(H)$ and $R$ is strong Feller, we have that 
  $\tilde f_{i,\ell}\in UC^\infty_b(H)$ for every $i,\ell\in\enne$.
  Moreover, 
  since $R$ extends to a strongly continuous semigroup on $L^2(H,\tilde \mu)$, one has
  by the dominated convergence theorem that 
  \[
  \lim_{\ell\to\infty}\|\tilde f_{i,\ell}-T_i(\tilde g)\|_{L^2(H,\tilde \mu)}=0 \quad\forall\,i\in\enne,\qquad
  \lim_{i\to\infty}\|T_i(\tilde g) - \tilde g\|_{L^2(H,\tilde \mu)}=0,
  \]
  so the conclusion follows trivially.
\end{proof}

\section*{Acknowledgements}
The authors are members of Gruppo Nazionale per l'Analisi Matematica, la Probabilit\`a 
e le loro Applicazioni (GNAMPA), Istituto Nazionale di Alta Matematica (INdAM), 
and gratefully acknowledge financial support 
through the project CUP\_E55F22000270001.
The research of the first-named author has been performed in the framework of the MIUR-PRIN Grant 2020F3NCPX ``Mathematics for industry 4.0 (Math4I4)''.






\end{document}